\documentclass[aos]{imsart}

\RequirePackage{amsthm,amsmath,amsfonts,amssymb}
\RequirePackage[numbers,sort&compress]{natbib}
\RequirePackage[colorlinks,citecolor=blue,urlcolor=blue]{hyperref}
\RequirePackage{graphicx}

\usepackage{amsthm}
\usepackage{hyperref}
\hypersetup
{colorlinks = true,
linkcolor = red, 
anchorcolor = red, 
citecolor = blue, 
filecolor = blue,
urlcolor = blue}
\usepackage{hypcap}
\usepackage{cleveref}
\usepackage{lipsum}
\usepackage{enumerate}
\usepackage{enumitem}
\usepackage{amsmath,amsfonts,amssymb,mathrsfs, amscd,amsthm,amsbsy,amsxtra,bbm,bm, epsf,calc,comment, xcolor}
\usepackage[toc,page]{appendix}
\usepackage{color}
\usepackage{datetime}
\usepackage{latexsym}
\usepackage[english]{babel}
\usepackage{graphicx}
\usepackage{epsfig}
\usepackage{dsfont}
\usepackage{tikz}
\usepackage{algorithm}
\usepackage{algpseudocode}

\startlocaldefs
\theoremstyle{plain}

\newtheorem{theorem}{Theorem}[section]
\newtheorem{lemma}[theorem]{Lemma}
\newtheorem*{lemma*}{Lemma}
\theoremstyle{remark}
\newtheorem{definition}[theorem]{Definition}

\endlocaldefs

\renewcommand{\vec}[1]{\boldsymbol{#1}}
\newcommand{\norm}[1]{\lVert #1 \rVert}

\definecolor{mygreen}{rgb}{0.1,0.75,0.2}

\DeclareMathOperator*{\esssup}{ess\, sup}

\DeclareSymbolFont{bbold}{U}{bbold}{m}{n}
\DeclareSymbolFontAlphabet{\mathbbold}{bbold}

\DeclareMathOperator{\dist}{dist}

\numberwithin{equation}{section}

\newtheorem{corollary}[theorem]{Corollary}
\newtheorem{assumption}[theorem]{Assumption}

\theoremstyle{remark}
\newtheorem{remark}[theorem]{Remark}

\begin{document}

\begin{frontmatter}
\title{Optimal sequencing depth for single-cell RNA-sequencing in Wasserstein space}
\runtitle{Optimal sequencing depth in scRNA-seq}

\begin{aug}
\author[A]{\fnms{Jakwang}~\snm{Kim}\ead[label=e1]{jakwangkim@cuhk.edu.cn}}\footnote{This work was done when this author was at department of mathematics of University of British Columbia.},
\author[B]{\fnms{Sharvaj}~\snm{Kubal}\ead[label=e2]{sharvaj@math.ubc.ca}}
\and
\author[B]{\fnms{Geoffrey}~\snm{Schiebinger}\ead[label=e3]{geoff@math.ubc.ca}}
\address[A]{School of Data Science, The Chinese University of Hongkong, Shenzhen, Guangdong, China \printead[presep={,\ }]{e1}}
\address[B]{Department of Mathematics, University of British Columbia, Vancouver, British Columbia, Canada \printead[presep={,\ }]{e2,e3}}
\end{aug}

\begin{abstract}
How many samples should one collect for an empirical distribution to be as close as possible to the true population? This question is not trivial in the context of single-cell RNA-sequencing. With limited sequencing depth, profiling more cells comes at the cost of fewer reads per cell. Therefore, one must strike a balance between the number of cells sampled and the accuracy of each measured gene expression profile. In this paper, we analyze an empirical distribution of cells and obtain upper and lower bounds on the Wasserstein distance to the true population. Our analysis holds for general, non-parametric distributions of cells, and is validated by simulation experiments on a real single-cell dataset.
\end{abstract}

\begin{keyword}[class=MSC]
\kwd[Primary ]{62G05}
\kwd{49Q22}
\kwd[; secondary ]{62D99}
\end{keyword}

\begin{keyword}
\kwd{RNA sequencing}
\kwd{optimal transport}
\kwd{Wasserstein distance}
\kwd{optimal depth}
\kwd{empirical distribution}
\kwd{nonparametric inference}
\end{keyword}

\end{frontmatter}
\section{Introduction}
\label{sec: Introduction}

The recent success of single-cell RNA sequencing (scRNA-seq) technologies \cite{kleinDropletBarcodingSingleCell2015, macoskoHighlyParallelGenomewide2015} has resulted in massive, high-dimensional datasets, which invite interesting statistical questions. 
There has been intensive work on clustering of cells \cite{traagLouvainLeidenGuaranteeing2019, grabskiSignificanceAnalysisClustering2023}, transformations and variance stabilization \cite{ahlmann-eltze_ComparisonTransformationsSinglecell_2023}, trajectory inference \cite{schiebingerOptimalTransportAnalysisSingleCell2019, qiuSystematicReconstructionCellular2022, lavenant_MathematicalTheoryTrajectory_2023} and manifold learning and denoising \cite{vandijkRecoveringGeneInteractions2018, huangSAVERGeneExpression2018}. 
Mathematically, the basic input in each of these tasks is an empirical distribution of cells; these are the cells sampled by scRNA-seq. 
However, contrary to the ordinary setting of high-dimensional statistics, the support of the empirical distribution is also noisy. This is because cells' gene expression profiles are measured, through sequencing, with finitely many \lq reads\rq.

A \lq read\rq~is the fundamental unit of information in DNA sequencing, and in RNA sequencing for that matter. 
The relative abundances of gene transcripts, which characterize cell type, can be measured by extracting RNA from the cell, 
and reading the sequence of nucleotide letters (ACTG) using a DNA sequencing machine. Each \lq read\rq~documents the presence of one molecule of RNA in one cell. 
Early forms of single-cell RNA sequencing produced high-fidelity measurements of gene expression by profiling small numbers of cells, with a large number of reads per cell. 
Droplet-based technologies made it possible to profile vastly more cells, but with fewer reads per cell. 
It soon became clear that this was beneficial for identifying rare cell types in complex populations of cells~\cite{jaitinMassivelyParallelSingleCell2014, shalekSinglecellRNAseqReveals2014, pollenLowcoverageSinglecellMRNA2014, streetsHowDeepEnough2014, bacher2016design, haque2017practical, ecker2017brain, dal2019design}.

The natural experimental design question is the following:
given a finite budget of reads, how many cells should be profiled so that the empirical distribution is as close as possible to the true population? 
As illustrated in \Cref{fig:concept-schematic}, profiling fewer cells allows each expression profile to be measured more accurately (with more reads per cell), but the full population is not well explored. On the other hand, profiling more cells gives an empirical distribution with larger support, but with more noise in the support which can blur out the finer features of the population distribution such as the two clusters (the peaks in the shaded contour plots).

\begin{figure}
    \centering
    \includegraphics[width=1.0\textwidth]{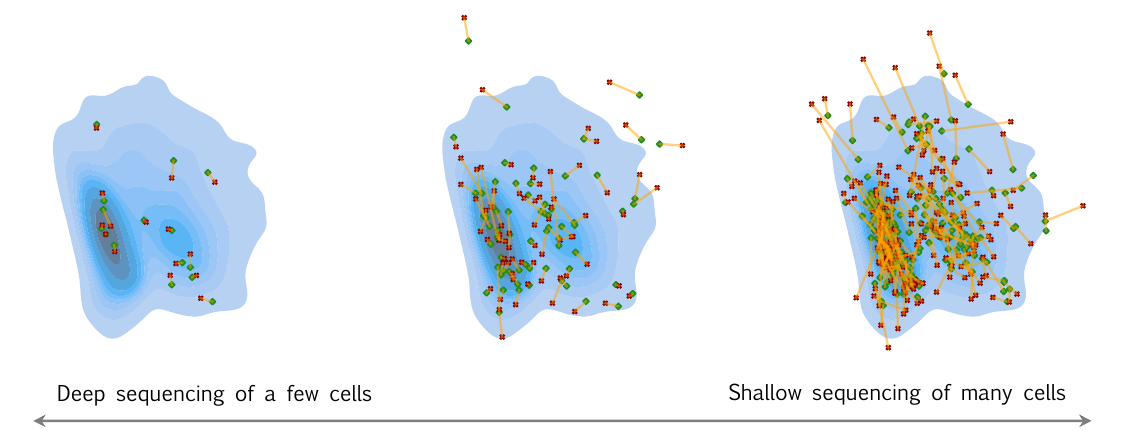}
    \caption{\textbf{The trade-off between deep and shallow sequencing}. In this schematic, we have three panels illustrating the effect of different sequencing depths. In each panel, the shaded contours describe the unknown distribution of gene expression in a population of cells. The green dots represent the true gene expressions of sampled cells, whereas the red markers indicate the corresponding measured gene expressions obtained via the sequencing process. The yellow lines match a cell's true and measured gene expressions.    
    When the number of cells is very low (left panel), each cell receives a high proportion of the budget of reads, leading to low measurement noise. However, the population distribution is not well-captured. Conversely, when the number of cells is very high (right panel), each cell receives a small fraction of the budget, resulting in very noisy measurements.
    The central panel depicts an intermediate case that balances these two effects.
    }
    \label{fig:concept-schematic}
\end{figure}

Due to its importance, there have been a lot of works \cite{jaitinMassivelyParallelSingleCell2014, shalekSinglecellRNAseqReveals2014, pollenLowcoverageSinglecellMRNA2014, streetsHowDeepEnough2014, heimbergLowDimensionalityGene2016, wangGeneExpressionDistribution2018, zhang2020determining} which try to answer this question with a certain assumption. In this work, we analyze this experimental design question in a very general setting.  
We leverage \emph{optimal transport}, more precisely \emph{Wasserstein distance}, to evaluate the discrepancy between between the noisy empirical distribution obtained by RNA sequencing and the unknown ground truth. There are many advantages of using optimal transport. First, it is well defined for the comparison between sufficiently general probability distributions. Moreover, Wasserstein distance takes into account \emph{not vertical but horizontal difference between distributions}, which displays geometric information of distributions. Most classical divergences, for example, Kullback–Leibler divergence, or the total variation distance are useless to understand the rate of convergence of empirical distributions. Wasserstein distance, on the other hand, is useful for it since it measures the cost of transporting mass from one to other place, hence the difference of distributions horizontally. We refer the readers to \cite{OT_for_applied}[Figure 5.1]. Second, Wasserstein geometry naturally involves \emph{gradient flow}, the evolution of distributions~\cite{mccann1994convexity, McCann1997Adv, Benamou_brener2000, otto2001geometry}. Since the recent demands of trajectory inference in single-cell RNA-sequencing data have emerged~\cite{schiebingerOptimalTransportAnalysisSingleCell2019, qiuSystematicReconstructionCellular2022, lavenant_MathematicalTheoryTrajectory_2023}, Wasserstein geometry and Wasserstein gradient flows are the state-of-art mathematical tool for this field. Lastly, Wasserstein distance is \emph{computationally tractable.} In his seminar work, \citet{cuturi2013sinkhorn} proposed an almost linear-complexity approximation algorithm, the so-called Sinkhorn algorithm or entropic optimal transport, which allows practitioners to use it in reality.

In this flexible framework, we are able to obtain results for general distributions of cells, without any assumptions on the distribution of interest. 
Moreover, because we focus on the empirical distribution directly, and not on any specific estimator, we hope that our results can inform broad estimation tasks.  
Our result depends on the sparsity of gene expression and the intrinsic low dimensionality of its distribution explicitly, which explains the role of those factors in RNA sequencing in a clear way. Furthermore, we conduct simulation experiments on a real single-cell dataset to validate and support the theory.

\subsection{Setup and contributions}\label{subsec: setup}
In this section, we introduce the mathematical notation and the setting of the problem, and state our contribution. Throughout this paper, $d$, $n$ and $m$ denote the number of genes, i.e. the dimension of gene expression vector, the number of cells, and the number of reads, respectively. In the sequel we use big-$O$ notation as usual, $f(n) \sim g(n)$ if $f$ and $g$ have the same first order growth, and $f(n) \lesssim g(n)$ if $f(n)$ is less than $g(n)$ in terms of variables of interest while other universal constants are fixed.

We model the measurement process with two steps. The first step is the selection of $n$ cells from the population. This will give us an empirical distribution of cells $\mu_n$, defined precisely in~\eqref{eq: true empirical mu n} below. Then, in the second step, we measure the gene expression of each cell through sequencing to obtain a noisy empirical distribution $\widehat{\mu}_n(m)$, defined precisely in~\eqref{eq: noisy empirical mu hat} below.

In sequencing, only the relative abundance of gene products matters. Therefore, we represent gene expression profiles of cells in relative terms. We denote a cell's gene expression vector by $X=(x_1, \dots, x_d) \in \Delta_{(d-1)}$ where $\Delta_{(d-1)}$ is the $(d-1)$-dimensional simplex. A population of cells is represented by a probability measure $\mu$ on $\Delta_{d-1}$, i.e., $X$ is distributed according to $\mu$. In practice, despite the large $d$, the distribution $\mu$ enjoys nice low dimensional structure, quantified here by an {\em intrinsic dimension} $k$ which will be defined rigorously later (see \Cref{def: Wasserstein dimension} and \Cref{rmk: intrinsic dimension}). Our goal is to obtain a distribution which approximates $\mu$.

In the first step of the measurement process, we sample $n$ cells from the population. This gives us i.i.d. samples from $\mu$ which we denote by $\vec{X}:=\{X_1, \dots, X_n\} \subseteq \Delta_{(d-1)}$. We denote the unknown true empirical distribution by
\begin{align}\label{eq: true empirical mu n}
    \mu_n:= \frac{1}{n} \sum_{i=1}^n \delta_{X_i}.   
\end{align}

In the second step of the measurement process, we consider two related models for the sequencing procedure. In both cases, cell expression profiles are generated multinomially from the true expression profile. The two models differ in how many reads are selected per cell. 
In the first model, cell sampling frequency $U \in \mathbb{R}_{>0}$ can depend on gene expression profile $X$. In the second model, $X$ and $U$ are independent. Model 1 accounts for the fact that larger cells contain more molecules of RNA, whereas model 2 accounts for the cell-type independent factors such as RNA capture efficiency. 
We use $\mu_{(X,U)}$ to denote the joint distribution of $U$ and $X$ for the first model. When we write $\mu$ with no subscript, it denotes the (marginal) distribution of $X$.

Given cells $X_1,\ldots,X_n\in\Delta_{d-1}$ with sampling frequencies $U_1, \dots, U_n\in{\mathbb{R}_{>0}}$, we sample $m$ i.i.d. reads to obtain noisy expression profiles $\widehat{X}_1, \dots, \widehat{X}_n \in \Delta_{(d-1)}$ as follows: 
\[
    \widehat X_i \sim \frac{1}{m_i} \text{Multinomial}(X_i,m_i),
\]
where $(m_1,\ldots,m_n)\sim \text{Multinomial}(\vec{u}, m)$, and $\vec{u}=(u_1, \dots, u_n)$ such that $u_i$'s are normalized sampling frequencies, i.e., $u_i:= \frac{U_i}{\sum_{j} U_j}$, for $i = 1, \dots, n$. For the case that $m_i=0$, which would happen with small but positive probability, we choose $\widehat{X}_i$ arbitrarily from $\Delta_{d-1}$. Finally, the \emph{noisy empirical distribution} is obtained as
\begin{align} \label{eq: noisy empirical mu hat}
    \widehat{\mu}_n(m):= \frac{1}{n} \sum_{i=1}^n \delta_{\widehat{X}_i}.   
\end{align}
Note that variants of such a multinomial sequencing model are standard in the literature~\cite{townesFeatureSelectionDimension2019, heimbergLowDimensionalityGene2016}, and are closely related to the popular Poisson model~\cite{wangGeneExpressionDistribution2018, zhang2020determining}.

Then, the goal is to answer two questions on $\widehat{\mu}_n(m)$, the output of shallow sequencing, in order to promise the closeness of $\widehat{\mu}_n(m)$ to $\mu$: 
\begin{enumerate}
    \item[(I)] given $n$ cells, how many reads at least should we use to obtain good $\widehat{\mu}_n(m)$,
    \item[(II)] given $m$ reads, how many cells at most can we use to obtain the best $\widehat{\mu}_n(m)$.
\end{enumerate}

In \Cref{thm: main result 1} and \Cref{thm: lower bound}, we resolve question (I) by showing that $m \gtrsim \frac{n}{\varepsilon^2}$ is sufficient to guarantee that Wasserstein distance between $\widehat{\mu}_n(m)$ and $\mu_n$ is less than $\varepsilon$, and the expected Wasserstein distance between $\widehat{\mu}_n(m)$ and $\mu_n$ is lower bounded by $\frac{n}{m}$ in general. Built from these results, we show that $n \sim m^{1 - \frac{2}{k+2}}$ is optimal in \Cref{thm: main result 2}, \Cref{cor: main result 2} and \Cref{cor: full-lower-bd}, which solves question (II). In particular, our results depend on the sparsity and the intrinsic dimension explicitly, which illustrates their effectiveness to approximation.





\subsection{Comparison to related work}
\label{subsec: prior-work}
The power of shallow sequencing was first explored by \citet{jaitinMassivelyParallelSingleCell2014} and \citet{shalekSinglecellRNAseqReveals2014} as they propelled single-cell sequencing technologies towards high throughput regimes. Their experiments were followed by a comprehensive analysis~\cite{pollenLowcoverageSinglecellMRNA2014} and a commentary~\cite{streetsHowDeepEnough2014}, establishing the validity and benefits of shallow sequencing for various downstream tasks. In particular, \citet{pollenLowcoverageSinglecellMRNA2014} demonstrated that shallow sequencing (up to a certain level) is sufficient for cell-type classification and clustering, even when subgroup differences are subtle. Additionally, they tracked the performance of principal component analysis (PCA), showing that PCA eigenvectors and sample scores correlate strongly between deep- and shallow-sequenced versions of the same dataset.

Variants of this experimental design question have been analyzed ~\cite{heimbergLowDimensionalityGene2016,zhang2020determining}. \citet{heimbergLowDimensionalityGene2016} analyzed the error in  {\em principal components} incurred with shallow sequencing. \citet{masoeroMoreLessPredicting2022} have addressed a related but different experimental design question for variant calling in genome sequencing.

The initial work mentioned above was followed up by \citet{heimbergLowDimensionalityGene2016}, whose key theoretical contribution involves quantifying the difference between the PCA eigenvectors of a "ground truth" gene expression dataset and the corresponding PCA eigenvectors of its shallow-sequenced version. In terms of our framework, these are the eigenvectors of the covariance matrices $\Sigma_{n}$ and $\widehat{\Sigma}_{n}(m)$ of the true empirical distribution $\mu_{n}$ and the noisy empirical distribution $\widehat{\mu}_{n}(m)$ respectively. In such a setting, \citet[Eq. 2]{heimbergLowDimensionalityGene2016} provide a read depth calculator for optimal budget allocation. Let $v_{i}\in\mathbb{R}^d$ and $\lambda_{i}\geq 0$ be the $i$-th eigenvector and eigenvalue of $\Sigma_{n}$ respectively, and let $\widehat{v}_{i} \in \mathbb{R}^d$ be the $i$-th eigenvector of $\widehat{\Sigma}_{n}(m)$. They prescribe the allocation
\begin{equation}\label{eq: heimberg allocation}
    \frac{m}{n} \sim \frac{\kappa^2}{n\lambda_{i}\varepsilon^2} 
\end{equation}
where $\varepsilon>0$ is the desired error such that $\|v_{i}-\widehat{v}_{i}\|_{2}\leq \varepsilon$, and $\kappa$ is a constant that can be estimated from existing data. The allocation \eqref{eq: heimberg allocation} is interesting since the number of cells $n$ cancels out, having no direct effect on the eigenvector error.
Nevertheless, \eqref{eq: heimberg allocation} does not address our problem of estimating the \emph{population} distribution $\mu$ or its associated quantities, since $v_{1}, \dots, v_{d}$ are the principal directions of the \emph{empirical} distribution $\mu_{n}$ and not of the population $\mu$. In that direction, it may be more meaningful to compare $\widehat{\Sigma}_{n}$ with the population covariance $\Sigma$ (i.e. covariance of $\mu$).

Estimating population distributions in this context was first addressed by \citet{wangGeneExpressionDistribution2018} and \citet{zhang2020determining}, and both relied on empirical Bayes (EB) deconvolution techniques of \citet{efronEmpiricalBayesDeconvolution2016}.
\citet{zhang2020determining} 
established minimax rates of $\Theta(\frac 1 m)$ for various estimation problems including 1-dimensional marginals, moments, and pairwise moments.
They also investigated optimal sequencing budget allocation for the number of reads $m$ and cells $n$, and showed that the minimax rate can be achieved by empirical Bayes estimators as long as
\begin{equation}\label{eq: tse allocation}
    \frac{m}{nd}=\Theta(1) 
\end{equation}
(see \cite[Theorems 1 and 2]{zhang2020determining}). However, their results do not cover estimation of high-dimensional, non-parametric distributions.

We study the optimal allocation problem in a high-dimensional, non-parametric setting for the first time to our knowledge. 
Our allocation $n \sim \left( \frac{ m}{ \mathbb{E} \Arrowvert X \Arrowvert_0 } \right)^{1 - \frac{2}{k + 2}}$ (see \Cref{cor: main result 2}) yields the rate 
$$
    \mathbb{E}W_{p}(\widehat{\mu}_{n}(m), \mu) \leq O \left( \left(\frac{\mathbb{E} \norm{X}_0 }{m}\right)^{\frac{1}{2+k}} \right)
$$ 
which explicitly incorporates the intrinsic dimension $k$ and the average sparsity $\mathbb{E} \Arrowvert X \Arrowvert_0$. Moreover, we note that the linear allocation \eqref{eq: tse allocation} of \citet{zhang2020determining} is similar to our result when $k$ is large, but it
is not quite optimal in our setting; it can suffer from a non-diminishing error  $\mathbb{E}W_{p}(\widehat{\mu}_{n}(m), \mu) = \Theta(1)$ under certain assumptions (see \Cref{cor: full-lower-bd}).

\subsection{Organization}
In \Cref{sec:mainresults}, some background of optimal transport, stochastic dominance and multinomial sampling (\Cref{sec:preplim}) and our main results (\Cref{sec:theorems}) will be provided. In \Cref{sec:Empiricalresults}, we will give an empirical evidence based on the real data set of \citet{schiebingerOptimalTransportAnalysisSingleCell2019}. Proofs of main results will be presented in \Cref{sec:proofs}. Lastly, we will summarize our contribution and discuss some future directions in \Cref{sec:conclusion}.

\section{Main results}\label{sec:mainresults}

\subsection{Preliminaries}\label{sec:preplim}
In the present paper, Wasserstein distance is the quantity to measure the goodness of the noisy empirical distribution.

Let's discuss the history of optimal transport very briefly. Optimal transport(OT) has a long story and is an interdisciplinary field connecting to both pure mathematics and applied sciences. Proposed by Monge~\cite{monge1781memoire}, optimal transport was developed by Kantorovich~\cite{kantorovich1942translocation, kantorovich1948problem} developed a standard framework for analyzing optimal transport, for which he received Nobel prize in economic sciences. Since then, OT has related to PDE, geometry and probability~\cite{Brenier91, mccann1994convexity, Caffarelli_contraction, Benamou_brener2000, Otto2000GeneralizationOA, otto2001geometry, Lott2004RicciCF, Figalli2010Invent, Kim2010JEMS, Figall2011duke, Caffarelli2013Reine, Figalli2015CMP, GOZLAN20173327}. Recently, 
OT has also played a central role in finance, statistics and machine learning: model free approach in mathematical finance and its statistical inference~\cite{hobson2012robust, Beiglb_Juillet16, kim2024convexorder}, 
optimization of mean field two-layer neural networks~\cite{mei2018mean, chizat2018global, sirignano2020mean, Rotskoff2022CPAM}, diffusion model~\cite{de2021diffusion, wang2021deep, hamdouche2023generative}, variational inference~\cite{lambert2022variational, diao2023forward} and adversarial training~\cite{pydi2021the, Trillos2020AdversarialCN, jakwang2023JMLR}. 

Given spaces $\mathcal{X}$ and $\mathcal{Y}$, and ``cost function'' $c: \mathcal{X} \times \mathcal{Y} \to \mathbb{R} \cup \{-\infty, +\infty\}$, the \emph{Monge problem} (named for a special case proposed  in \cite{monge1781memoire}) is to find a map $T : \mathcal{X} \to \mathcal{Y}$ which achieves the minimum of the following optimization problem:
\[
    \inf_{(T)_{\#}(\mu) = \nu} \int_{\mathcal{X}}c(x, T(x) )d\mu(x)
\]
where $(T)_{\#}(\mu)$ is a pushforward measure, i.e. $(T)_{\#}(\mu)(B) := \mu \left( T^{-1}(B) \right)$ for any measurable function $T$ and measurable subset $B$. It is well known that, however, such a $T$ may not exist even for nice $c$. For the well-definedness of the problem, Kantorovich extended the solution space of the Monge problem to the set of couplings of marginals \cite{kantorovich1942translocation, kantorovich1948problem}. The problem becomes an infinite dimensional linear optimization problem, which is defined as 
    \[
        \min_{\pi \in \Pi(\mu, \nu)} \int_{\mathcal{X} \times \mathcal{Y}} c(x,y) d\pi(x,y).
    \]
If the spaces $\mathcal{X}$ and $\mathcal{Y}$ are Polish(complete and metrizable), then $\Pi(\mu, \nu)$ is non-empty and compact with respect to weak topology. For nonnegative and lower semi-continuous $c$, a minimizer of Kantorovich's problem always exists \cite[Theorem 4.1]{Oldandnew}.

The Kantorovich formulation provides a rich mathematical framework: it can be leveraged to define a distinguished family of metrics on the space of probability measures. Let $\mathcal{X}=\mathcal{Y}$ be a Polish (complete metric) space. The $p$-Wasserstein distance, or Kantorovich-Rubinstein distance, is defined over $\mathcal{P}_p(\mathcal{X})$, as follows:
\[
    {W}_p(\mu, \nu):= \min_{\pi \in \Pi(\mu, \nu)} \left( \int_{\mathcal{X} \times \mathcal{X}} \dist(x,y)^p d\pi(x,y) \right)^{\frac{1}{p}}.
\]
Let $\mathcal{P}_p\left(  \mathcal{X} \right)$ denote the set of probability measures over $ \mathcal{X}$ with finite $p$-th order moment. It is well known that $\mathcal{P}_p(\mathcal{X})$ equipped with $W_p$ is a metric space. In addition, $(\mathcal{P}_2(\mathcal{X}), {W}_2)$ carries Riemannian structure \cite{otto2001geometry}, which generates plentiful implications in many fields in mathematics. Recently, even this space plays numerous roles in many different sciences, especially machine learning and statistics: see \citet{computation_OT}.

In the sequel, we will utilize recent progresses concerning the rate of convergence of empirical distributions from \citet{Niles-Weed_Bach}. To state them rigorously, we need to introduce Wasserstein dimension.

\begin{definition}\cite{Niles-Weed_Bach}
\label{def: Wasserstein dimension}
Let $(\mathcal{F}, || \cdot||)$ be a (semi)normed space. Given a set $S \subseteq \mathcal{F}$, the $\varepsilon$-covering number of $S$, denoted by $\mathcal{N}(\varepsilon, S, || \cdot ||)$, is the minimum $m$ such that there exist $m$ closed balls $B_1, \dots, B_m$ of diameter $\varepsilon$ such that $S \subseteq \cup_{i=1}^m B_i$. The centers of the balls need not belong to $S$. The $\varepsilon$-dimension of $S$ is defined as
\[
    d(\varepsilon; S) := \frac{\log \mathcal{N}(\varepsilon, S, || \cdot ||)}{-\log \varepsilon}.
\]
Given a probability measure $\mu$ on a metric space $\mathcal{F}$, the $(\varepsilon, \tau)$-covering number is
\[
    \mathcal{N}(\varepsilon, \tau; \mu) := \inf\{ \mathcal{N}(\varepsilon, S, || \cdot ||) : \mu(S) \geq 1 - \tau \}
\]
and the $(\varepsilon, \tau)$-dimension is
\[
    d(\varepsilon, \tau; \mu) := \frac{\log  \mathcal{N}(\varepsilon, \tau; \mu)}{-\log \varepsilon}.
\]
The upper and lower Wasserstein dimensions are respectively,
\begin{align*}
    d^*_p(\mu) &:= \inf \left\{ s \in (2p, \infty) : \limsup_{\varepsilon \to 0} d(\varepsilon, \varepsilon^{\frac{sp}{s - 2p}}; \mu) \leq s \right\},\\
    d_*(\mu) &:= \lim_{\tau \to 0} \liminf_{\varepsilon \to 0} d_{\varepsilon, \tau}(\mu).
\end{align*}
\end{definition}

\begin{remark}\label{rmk: intrinsic dimension}
In the present paper, roughly speaking, $k > \max\{ d^*_p(\mu), 2p\}$ will be regarded as the intrinsic dimension of $\mu$. There is a trade off in choosing such a $k$ to optimize the rates of convergence of empirical distributions in Wasserstein distances: see \cite[Proposition 5]{Niles-Weed_Bach}.
\end{remark}

We end this subsection to introduce $\ell_0$ norm, which is neither norm nor even pseudo-norm but is called norm conventionally. $\ell_0$ norm of $X$ is a critical parameter of the main results. For $X= (x_1, \dots, x_{d}) \in \mathbb{R}^{d}$, the $\ell_0$ norm of $X$ is defined as
\begin{equation}\label{eq: ell_0 norm}
    \Arrowvert  X \Arrowvert_0:= \sum_{j=1}^{d} \mathds{1}_{x_j \neq 0}.
\end{equation}
One can understand $\ell_0$ norm as the sparsity of a vector.

\subsection{Theorems}\label{sec:theorems}
Let $\Delta_{(d-1)}$, the $(d-1)$-dimensional simplex, be equipped with $\ell_q$ distance for any $1 \leq q \leq \infty$: for $X, X' \in \Delta_{(d-1)}$
\[
    \dist(X, X'):=\Arrowvert X - X'\Arrowvert _q =
    \begin{cases}
        \left( \sum_{j=1}^{d} |x_j - x'_j|^q \right)^{\frac{1}{q}} &\text{ if $1 \leq q < \infty$,}\\
        \max_{1 \leq j \leq d} |x_j - x'_j| &\text{ if $q=\infty$.}
    \end{cases}
\] 
Let us recall the sequencing procedure stated in \Cref{subsec: setup}. Let $\vec{X}=\{X_1, \dots, X_n\} \subseteq \Delta_{(d-1)}$ and $\vec{u}=(u_1, \dots, u_n) \in \Delta_{(n-1)}$ denote the set of i.i.d. cells (samples, or gene expressions) and the (normalized) weight vector, respectively. Noisy profiles $\widehat{X}_i$'s follow $\widehat{X}_i \sim m_i^{-1} \text{Multinomial}(X_i, m_i)$ where $(m_1, \dots, m_n) \sim \text{Multinomial}(\vec{u}, m)$. If some of $X_i$'s are never observed, we choose $\widehat{X}_i \in \Delta_{d-1}$ arbitrarily. Finally, the noisy empirical distribution, $\widehat{\mu}_n(m)$, is obtained based on $\widehat{X}_i$'s.

In order to ensure that the noisy empirical distribution is a good approximation, we need the assumption on $\vec{u}$. If most $u_i$'s are too small, for example $u_i=O(n^{-2})$, then it is unlikely to extract gene expression from those cells so that the noisy empirical distribution cannot be close to the truth. Therefore, it should be required that most of cells can be selected according to $\vec{u}$.

\begin{assumption}\label{assumption}
For any $n \in \mathbb{N}$, $\vec{u}$ is concentrated on $\{u \in \Delta_{(n-1)} : |\{i : u_i = o (n^{-1}) | = o(\sqrt{n \log n}) \}$: more precisely, there exists a deterministic constant $c_* > 0$ independent of $n$ such that for $n-o(\sqrt{n \log n})$ many $u_i$'s, it holds almost surely that
\begin{equation}\label{assumption : lower bound of u_i}
    u_i \geq \frac{c_*}{n}.
\end{equation}
\end{assumption}

The first theorem answers to question (I): given an error tolerance $\epsilon$ and $n$-cells, how many reads does it require to guarantee that Wasserstein distance between the noisy empirical distribution and the true empirical one is less than $\epsilon$.

\begin{theorem}\label{thm: main result 1}
Fix $p \in [1,2]$, and suppose that \Cref{assumption} holds. Then,
\begin{equation}\label{eq: expected W_p between noisy and true empirical dists}
    \mathbb{E} W_p(\widehat{\mu}_n(m), \mu_n) \leq  \sqrt{\frac{32 \mathbb{E}\norm{X}_0}{c_*}\,\frac{n}{m}} \wedge 2 + o \left( \sqrt{\frac{\log n}{n} } \right).
\end{equation}
Additionally, given $\epsilon \in (0,1)$, $n \in \mathbb{N}$  and $\alpha \in (0,1)$, if $m$ satisfies
\begin{equation}\label{eq: lower bound of m}
    m \geq \frac{ 32 ( 1+ \alpha)  n \mathbb{E}\Arrowvert X \Arrowvert_0 }{c_* \epsilon^2} - \frac{ 32 ( 1+ \alpha) \epsilon^{p-2}  n \mathbb{E}\Arrowvert X \Arrowvert_0 }{c_* }\sqrt{ \frac{\alpha \log n}{n} },
\end{equation}
then it holds that
\begin{equation}\label{eq: W_p between noisy and true empirical dists}
    W_p( \widehat{\mu}_n(m), \mu_n) \leq \epsilon
\end{equation}
with probability $ 1 - O\left( n^{-\frac{\alpha}{2}} \right)$. The hidden constants in $O\left( n^{-\frac{\alpha}{2}} \right)$ depend on $\alpha$, $p$, $\mathbb{E} \Arrowvert X \Arrowvert_0$ and $\esssup \Arrowvert X \Arrowvert_0$.   
\end{theorem}

Regarding \eqref{eq: lower bound of m}, a lower bound of $m$ is proportional to $\alpha, \mathbb{E} \Arrowvert X \Arrowvert_0, \epsilon^{-1}, c_*^{-1}, p^{-1}$ (obviously, as $n$ larger, $m$ should be larger). This theorem does not depend on any dimension since $\mu_n$ is an empirical distribution (discrete). Here, the most important parameter is $\mathbb{E} \Arrowvert X \Arrowvert_0$, the \emph{average sparsity of $X$}. Recalling \Cref{subsec: setup}, if many $X_i$'s are sparse, then a fewer reads are sufficient to approximate them well. As $n$ grows large, concentration phenomenon occurs so that most of $\Arrowvert X_i \Arrowvert_0$'s concentrate on $\mathbb{E} \Arrowvert X \Arrowvert_0$. For this reason, the average sparsity plays a role in order to determine a lower bound of the effective number of reads.

The inverse proportion of $\epsilon$ is usual but the dependence of $c_*$ needs to be explained. Gene expressions of $n - o(\sqrt{n \log n})$ many cells should be measured enough to achieve \eqref{eq: W_p between noisy and true empirical dists}. It is then necessarily satisfied that $ n - o(\sqrt{n \log n})$ many cells appear during the sequencing. Unless $c_*$ is $O(1)$, the event that $O(n)$ many cells are not observed enough would happen with non-negligible probability.

$\alpha$ determines the rate of the probability of the event of interest at which the probability of the event $W_p( \widehat{\mu}_n(m), \mu_n) > \epsilon$ decays to $0$ as $n \to \infty$. One can choose $\alpha$ such that it vanishes as $n$ goes $\infty$ while not too fast: for example $\alpha = O\left( \frac{\log \log n}{\log n} \right)$.

\Cref{thm: main result 1} tells us that $\mathbb{E}W_p(\widehat{\mu}_n(m), \mu_n) \lesssim \sqrt{\frac{n}{m}} + o\left( \sqrt{\frac{\log n}{n}} \right)$, suggesting that $m$ needs to grow faster than $n$ in order to obtain a diminishing Wasserstein distance between $\widehat{\mu}_n(m)$ and $\mu_n$. The following theorem provides a lower bound, confirming that this is indeed the case under certain assumptions.

\begin{theorem}\label{thm: lower bound}
Let $\mathrm{dist(\cdot,\cdot)}$ be the $\ell_{q}$ metric on $\Delta_{(d-1)}$ for $q \in [1,2]$, and assume uniform cell weights $\boldsymbol{u}=\left(\frac{1}{n}, \dots ,\frac{1}{n} \right)$. Then for any $p \geq 1$,
\begin{equation*}
\mathbb{E}W_{p}(\widehat{\mu}_{n}(m), \mu_{n}) \geq \frac{1-\mathbb{E}\|X\|_2^2}{4} \frac{n}{m}
\end{equation*}
provided $\mathbb{E}\|X\|_2^2 < 1$ and $m \geq 2n \log \frac{4}{1-\mathbb{E}\|X\|_2^2}$.  
\end{theorem}

The next theorem and its corollary answer question (II), which is the main theme of this paper: given the number of $m$ reads, how many cells can one use in order to ensure that the noisy empirical distribution is as close as possible to the true distribution $\mu$.

Let us recall a trade-off for approximating $\mu$ in RNA sequencing: on the one hand a larger $n$ (i.e. more cells) helps get a better approximation of $\mu$ as usual because it allows to explore the population more, but on the other hand, since $m$ is fixed, if $n$ is too large compared to $m$, gene expression of each cell cannot be measured precisely well so that the noisy empirical distribution is poorly qualified. Therefore, the optimal number of cells should be determined sophisticatedly by taking into account these contradictory aspects.

\begin{theorem}\label{thm: main result 2}
Fix $\alpha \in (0,1)$ and $p \in [1,2]$. Assume that there are some $k > \max\{ d^*_p(\mu) , 2p\}$ and $m \geq m_0$ for some $m_0=m_0 (c_*, \alpha, \mathbb{E} \Arrowvert X \Arrowvert_0 ) > 0$. Let such $m$ be fixed.

Under Assumption \ref{assumption}, if $n$ satisfies \eqref{eq: lower bound of m} given $m$, then $\widehat{\mu}_{n}(m)$ satisfies
\begin{equation}\label{eq: W_p between noisy empirical and the truth}
    W_p(\widehat{\mu}_{n}(m), \mu) \leq O \left(n^{-\frac{1}{k}} \right)
\end{equation}
with probability at least $1 - O\left(n^{-\frac{\alpha}{2}} \right)$. The hidden constants in $O \left(n^{-\frac{1}{k}} \right)$ and $O\left( n^{-\frac{\alpha}{2}} \right)$ depend on $k$ and $p$, and $\alpha$, $p$, $\mathbb{E} \Arrowvert X \Arrowvert_0$ and $\esssup \Arrowvert X \Arrowvert_0$, respectively.
\end{theorem}

The parameter $k$ can be regarded as the \emph{intrinsic dimension} of the support of $\mu$. In reality, the intrinsic dimension $k$ is much smaller than the ambient dimension. If the support of $\mu$ has a nice low dimensional structure, then given the same number of reads it is allowed to use more cells to approximate the population more, which helps get the noisy empirical distribution close to the truth. In other words, the low dimensional structure allows to leverage more cells to explore population more widely. 

While the intrinsic dimension is unknown a priori in practice, we discuss a heuristic approximation based on the result by \citet{JMLR:v23:21-1483} later in \Cref{loc:assessing_the_intrinsic_dimension}.

\begin{remark}\label{rmk: dimension and sparsity}
One must be careful about the relation between the sparsity of $X$ and the intrinsic dimension of $\mu$. It is possible that even if $\mu$ has a small intrinsic dimension, $X$ is not sparse. Suppose $\mu$ is concentrated on the diagonal, the span of $\{ (1, \dots, 1) \}$. In this case $d^*_p(\mu)=1$ but $\Arrowvert X \Arrowvert_0 = d$. On the other hand, however, if $\Arrowvert X \Arrowvert_0$ is concentrated at most $r \ll d$, then $\mu$ must exhibit a lower intrinsic dimensionality. In this sense, the sparsity of $X$ implies a lower intrinsic dimension of $\mu$, but the converse does not hold in general.    
\end{remark}

According to \Cref{thm: main result 2}, one obtains the following corollary in a straightforward way. Here, we ignore lower order terms, and focus on the parameters of interest only. The proof is omitted.

\begin{corollary}\label{cor: main result 2}
Fix $\alpha \in (0,1)$ and $p \in [1,2]$. Assume that there are some $k > \max\{ d^*_p(\mu) , 2p\}$ and $m \geq m_0$ for some $m_0=m_0 (c_*, \alpha, \mathbb{E} \Arrowvert X \Arrowvert_0 ) > 0$. Under \Cref{assumption}, if $n$ satisfies
\begin{equation}\label{eq: optimal n in practice}
    n \sim \left( \frac{ m}{ \mathbb{E} \Arrowvert X \Arrowvert_0 } \right)^{1 - \frac{2}{k + 2}},
\end{equation} 
then $\widehat{\mu}_n(m)$ achieves
\begin{equation}\label{eq: optimal Wasserstein distance}
    W_p(\widehat{\mu}_{n}(m), \mu)  \lesssim \left( \frac{ \mathbb{E} \Arrowvert X \Arrowvert_0}{ m } \right)^{\frac{1}{k + 2}}
\end{equation}
with probability at least $1 - O\left( \left( \frac{ m}{ \mathbb{E} \Arrowvert X \Arrowvert_0 } \right)^{-\frac{\alpha}{3}} \right)$.
\end{corollary}

It is worthwhile to mention that the order \eqref{eq: optimal Wasserstein distance} is almost optimal. If the support of $\mu$ has the intrinsic dimension $k$, then $W_p(\mu_n, \mu) = O \left( n^{-\frac{1}{k}} \right)$ with high probability~\cite[Theorem 1, Proposition 20]{Niles-Weed_Bach}. On the one hand, \eqref{eq: optimal Wasserstein distance} yields that
\begin{equation}\label{eq: aymptotic upper bound}
    \text{if $n \sim m^{\frac{k}{k+2}}$, then $W_p(\widehat{\mu}_{n}(m), \mu) \lesssim m^{-\frac{1}{k+2}}$}
\end{equation}
In this sense, \eqref{eq: optimal Wasserstein distance} achieves the almost optimal rate of convergence.

We want to emphasize that \eqref{eq: optimal Wasserstein distance} explicitly illustrates how the sparsity and the intrinsic dimension play a role in the optimal allocation, which the previous literature did not capture. It shows that Wasserstein distance increases as $k$ and $\mathbb{E} \Arrowvert X \Arrowvert_0$ increase, which perfectly coincides with our intuition. Trivially since larger $k$ and $\mathbb{E} \Arrowvert X \Arrowvert_0$ require to measure more genes, the Wasserstein distance cannot but increase whenever the number of reads is fixed.

Lastly, we conclude this section by providing a lower bound on $W_p(\widehat{\mu}_n(m), \mu)$, which follows directly from \Cref{thm: lower bound} and the bound $\mathbb{E}W_p(\mu_n, \mu) \leq O\left(n^{-\frac{1}{k}}\right)$ of \Cref{lemma: theorem Niles-Weed_Bach} (\citet[Theorem 1]{Niles-Weed_Bach}) by applying a triangle inequality twice. We omit the proof.

\begin{corollary}\label{cor: full-lower-bd}
Let $\mathrm{dist(\cdot,\cdot)}$ be the $\ell_{q}$ metric on $\Delta_{(d-1)}$ for $q \in [1,2]$. Let $k > \max\{d_p^*(\mu), 4\}$ and uniform cell weights $\boldsymbol{u}=\left(\frac{1}{n}, \dots ,\frac{1}{n} \right)$. Assume further that $\mathbb{E}\|X\|_2^2 < 1$ and $m \geq 2n \log \frac{4}{1-\mathbb{E}\|X\|_2^2}$. Then for any $p \geq 1$,
\begin{equation}\label{eq: lower bound optimal Wasserstein distance}
    \mathbb{E}W_p(\widehat{\mu}_n(m), \mu) \geq \max\left\{ O\left(\frac{n}{m} \right), O \left( n^{-\frac{1}{k}}\right) \right\} - \min\left\{ O\left(\frac{n}{m} \right), O \left( n^{-\frac{1}{k}}\right) \right\}
\end{equation}
\end{corollary}

\begin{remark}
Let $n=m^r$ and $0 < r < 1$. If the leading order of \eqref{eq: lower bound optimal Wasserstein distance} is $n^{-\frac{1}{k}}$, then $r \leq \frac{k}{k+1}$. Similarly, if the leading order of \eqref{eq: lower bound optimal Wasserstein distance} is $\frac{n}{m} = m^{r-1}$, $r \geq \frac{k}{k+1}$. Hence $r=\frac{k}{k+1}$ minimizes the leading order. It turns out that 
\begin{equation}\label{eq: aymptotic lower bound}
    \text{if $n \sim m^{\frac{k}{k+1}}$, then $\mathbb{E}W_p(\widehat{\mu}_n(m), \mu) \gtrsim m^{-\frac{1}{k+1}}$.}
\end{equation}
Focusing on the order only, comparing \eqref{eq: aymptotic lower bound} with \eqref{eq: aymptotic upper bound}, they almost match.
\end{remark}

\section{Empirical results}\label{sec:Empiricalresults}

\subsection{Experimental setup}
\label{loc:experimental_setup}
In this section, we empirically investigate the dependence of the optimal number of samples $n^{*}$ on the sequencing budget $m$ via simulation experiments on real single-cell datasets. Given a ground truth population distribution $\mu$ constructed from a real dataset, we generate $n$ i.i.d. samples followed by $m$ reads for various values of $n$ and $m$, yielding noisy empirical distributions $\hat{\mu}_{n}(m)$ (recall the sampling and sequencing procedures from \Cref{sec:preplim}). The resulting Wasserstein errors $W_{1}(\hat{\mu}_{n}(m),\mu)$ then reveal the desired dependence, demonstrating the validity of our theoretical results from \Cref{sec:theorems}. Throughout this section, we use the $\ell_1$ distance on the simplex $\Delta_{(d-1)}$ as our ground metric in computing $W_1$.

\subsubsection{Datasets}
\label{loc:experimental_setup.dataset}
We focus on a subset of the scRNA-seq \emph{cellular reprogramming} dataset from \citet{schiebingerOptimalTransportAnalysisSingleCell2019} in this section – particularly cells from the D15.5 Serum timepoints. Results from similar experiments on two additional datasets -- the \emph{sea urchin} dataset~\cite{massri_SingleCellTranscriptomicsReveals_2025} and the \emph{Arabidopsis} dataset~\cite{shahanSinglecellArabidopsisRoot2022} -- are deferred to \Cref{app: results-additional-datasets}. Further details about the specific data files and preprocessing tools are presented in \Cref{appsubsec: dataset-preprocessing}.

Each dataset here is available as a cells-by-genes counts matrix of reads recorded in sequencing experiments (more precisely, unique molecular identifier (UMI) reads \cite{kivioja_CountingAbsoluteNumbers_2012}). Denote this matrix as
\[
    \mathbf{C} = 
    \begin{bmatrix} 
    - & C_{1}^T & -  \\
    & \vdots &  \\
    - & C_{N}^T & - 
    \end{bmatrix} 
    \, \in \mathbb{N}_{0}^{N \times d},
\]
whose rows correspond to the $N$ cells measured. Note that although this matrix was obtained experimentally, which necessarily corresponds to an `empirical' distribution in reality, we treat it as the ground truth population in our simulations. Particularly, we construct the distribution $\mu$ as
\begin{equation}
\label{eq: constructed-ground-truth-mu}
\mu = \frac{1}{N} \sum_{\ell=1}^N \delta_{C_{\ell} / \|C_{\ell}\|_{1}} \in \mathcal{P}(\Delta_{(d-1)}).
\end{equation}
Note that $\mu$ defined in such a manner is an atomic measure; its discrete nature enables easy generation of (pseudorandom) samples and computation of Wasserstein distances $W_{1}(\hat{\mu}_{n}(m),\mu)$. 
For $\mu$ constructed from the \emph{reprogramming} dataset considered in this section, we have $d=1000$, $N=4851$ and $\mathbb{E}_{P \sim \mu}\|X\|_0 \simeq 175$.

\subsubsection{Cell weights}
\label{loc:size_factors}
Recall that the noisy sequencing process stated in \Cref{subsec: setup} depends on the normalized cell weights $\boldsymbol{u}=(u_{1}, \dots, u_{n}) \in \Delta_{n-1}$, which, as discussed earlier, may or may not be independent of the gene expression vectors $\{ X_1, \dots, X_n \}$. A simple scenario is that of uniform cell weights $\boldsymbol{u}=\left(\frac{1}{n}, \dots ,\frac{1}{n} \right)$, which we investigate first (\Cref{loc:results_on_the_wot_dataset}, \Cref{subsec: empirical-res-various-intrinsic-dim}). A general formulation considers the joint distribution $(X,U) \sim \mu_{(X,U)}$, and this is studied in \Cref{subsec: size-factors}.

\subsubsection{Assessing the intrinsic dimension}
\label{loc:assessing_the_intrinsic_dimension}

The theoretical results from \Cref{sec:theorems} depend on the intrinsic dimension $k$ (related to the upper Wasserstein dimension) of $\mu$ (see \Cref{def: Wasserstein dimension} and \Cref{rmk: intrinsic dimension}). 
Because the Wasserstein dimensions can be hard to compute in practice, we rely on a heuristic approximation following the results of \citet{Niles-Weed_Bach} and ideas from \citet{JMLR:v23:21-1483}. Noting that sufficiently regular probability measures $\mu$ display
\begin{equation*}
\Omega(n^{-1/\underline{k}}) \leq \mathbb{E}\,W_{1}(\mu, \mu_{n}) \leq O(n^{-1/\overline{k}})
\end{equation*}
(where $\mu_{n}$ is the $n$-sample (true) empirical distribution) for any $\underline{k} < d_{*}(\mu) =d^*(\mu) < \overline{k}$, one can use the rate of decay of $\mathbb{E}\,W_{1}(\mu, \mu_{n})$ to estimate $k$. More precisely, for various values of $n$, we repeatedly draw i.i.d. samples from $\mu$ to form $\mu_{n}$, and compute $W_{1}(\mu,\mu_{n})$ for each such trial. Then, we fit a straight line to the resulting set of pairs $(\log n, \log W_{1}(\mu,\mu_{n}))$, the slope of which approximates $-1 /k$. For the reprogramming dataset, this procedure gives $k=9.718$ (which is much smaller than the ambient dimension $d=1000$); see \Cref{appsubsec: population-statistics} and \Cref{fig:k-estimation-plot}.

In \Cref{subsec: empirical-res-various-intrinsic-dim} below, we conduct tests on lower-dimensional versions of $\mu$, with various intrinsic dimensions $k$, that we construct via non-negative matrix factorization of the normalized counts matrix.

\begin{figure}[H]
\centering
\includegraphics[width=\textwidth]{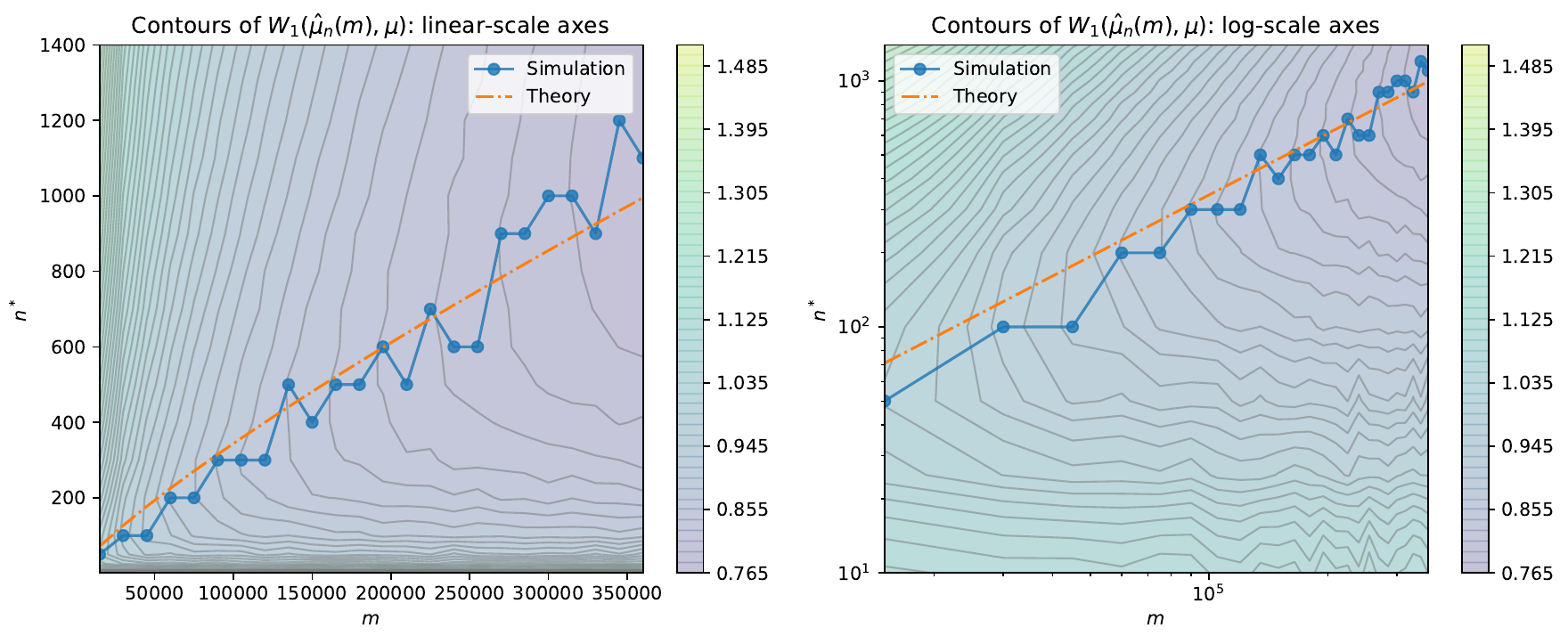}
\caption{\textbf{Optimal allocation ($m$-$n$ relationship) with uniform cell weights}. The contours and the colour scale show $W_{1}(\widehat{\mu}_{n}(m), \mu)$ averaged over ten (10) independent realizations of $\widehat{\mu}_{n}(m)$, for various pairs of $m$ and $n$ on a grid. The $m$ and $n$ axes are plotted in a linear scale on the left panel, and in a log scale on the right panel. The optimal values of $n^*$ for each $m$, as observed in this simulation, are plotted with blue solid dots. Our theoretically derived optimal allocation, with estimated intrinsic dimension $k = 9.718$}, is plotted as an orange dash-dot line.
\label{fig: wot unif size factors}
\end{figure}

\subsection{Results on the reprogramming dataset with uniform cell weights}
\label{loc:results_on_the_wot_dataset}
For various pairs of $n$ and $m$, we simulate the noisy empirical distribution $\widehat{\mu}_{n}(m)$ for ten (10) independent trials, and compute the mean Wasserstein error $W_{1}(\widehat{\mu}_{n}(m), \mu)$ over those trials. Contours of mean $W_{1}(\widehat{\mu}_{n}(m), \mu)$ over $m$-$n$ axes are plotted in \Cref{fig: wot unif size factors}. For each budget $m$ tested, we identify the number of cells $n=n^*$ that yields the lowest mean $W_{1}(\widehat{\mu}_{n}(m), \mu)$. This $m$-$n^*$ relationship, shown in \Cref{fig: wot unif size factors}, appears to be well captured by our theoretically derived optimal allocation (from \Cref{cor: main result 2})
\begin{align}
\label{eq: optimal n with C}
n \sim \left( \frac{Cm}{\mathbb{E}\|X\|_0} \right)^{1 - \frac{2}{2+k}}
\end{align}
where $C$ is a $\Theta(1)$ constant. We have used the value $C = 2$ in the figure. The {exponent $1-\frac{2}{2+k}$} shows up as the slope on the log-scale plot in \Cref{fig: wot unif size factors} (right panel).

\subsection{Results for various intrinsic dimensions}
\label{subsec: empirical-res-various-intrinsic-dim}

\begin{figure}[b]
    \centering
    \includegraphics[width=\textwidth]{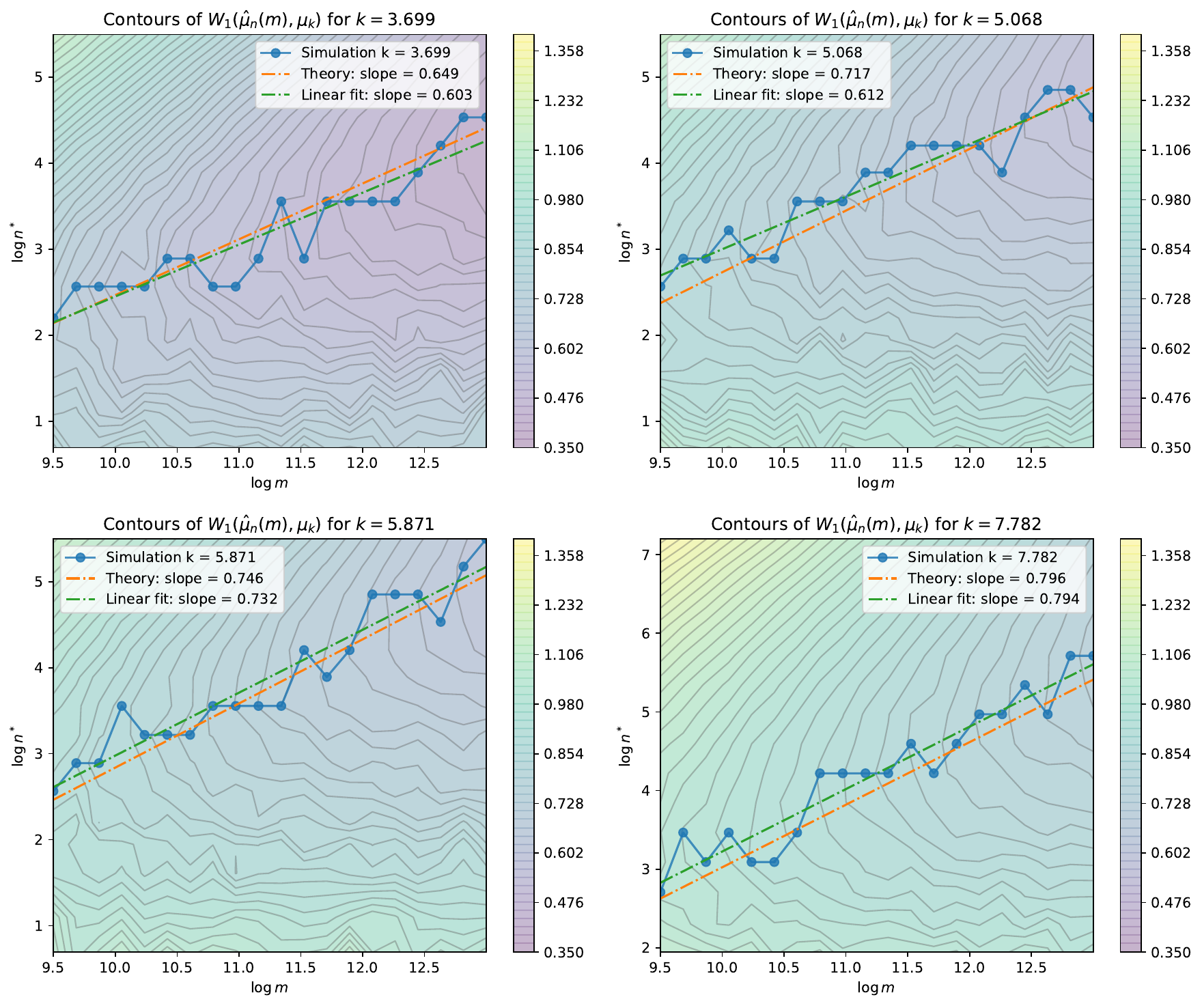}
    \caption{\textbf{Optimal allocations for different intrinsic dimensions $k$}. Each panel shows the optimal allocation results for a synthetic population distribution $\mu_k$ of a different intrinsic dimension $k$. The contours and the color scale show $W_{1}(\widehat{\mu}_{n}(m), \mu_k)$ averaged over ten (10) independent realizations of $\widehat{\mu}_{n}(m)$, for various pairs of $m$ and $n$. The optimal values $\log n^*$ for each $\log m$, as observed in the simulations, are plotted with blue solid dots. Linear least squares fits to these $\log m$ - $\log n^*$ relationships are plotted as green dash-dot lines. Finally, our theoretically derived optimal allocations are plotted as orange dash-dot lines.}
    \label{fig: nmf-synthetic}
\end{figure}

The optimal allocation is crucially dependent on the intrinsic dimension $k$. We investigate this further by constructing synthetic versions of the reprogramming dataset with different values of $k$ to produce synthetic population distributions $\mu_k$. Details on these constructions, and some statistics of the synthetic datasets, are presented in \Cref{appsubsec: lower-dim}. As before, the intrinsic dimensions in question are estimated via the heuristic from \Cref{loc:assessing_the_intrinsic_dimension}; see \Cref{fig:synthetic-k-estimation-plot}.

We repeat the experiment from \Cref{loc:results_on_the_wot_dataset} on the synthetic distributions $\mu_k$ for various values of $k$; the observed $\log m$ - $\log n^*$ relationships are presented in \Cref{fig: nmf-synthetic}. For a more quantitative assessment, we estimate the slope of each observed $\log m$ - $\log n^*$ relationship via a linear least squares fit (green dash-dot line), and track its dependence on $k$. The optimal allocation theory~\eqref{eq: optimal n with C} suggests that this slope be $1- \frac{2}{k+2}$ (orange dash-dot line). We set $C=2$ as in \Cref{loc:results_on_the_wot_dataset}.

We notice that the theory predicts the fitted slopes fairly well across the values of $k$ tested. Moreover, the different slopes observed in \Cref{fig: nmf-synthetic}, which are smaller than 1 in particular, demonstrate variably sublinear dependence of $n^*$ on $m$; this is different from the linear relationship discussed in \citet{zhang2020determining}.

\subsection{Results in the case of non-uniform cell weights}
\label{subsec: size-factors}

\begin{figure}[b]
    \centering
    \includegraphics[width=\textwidth]{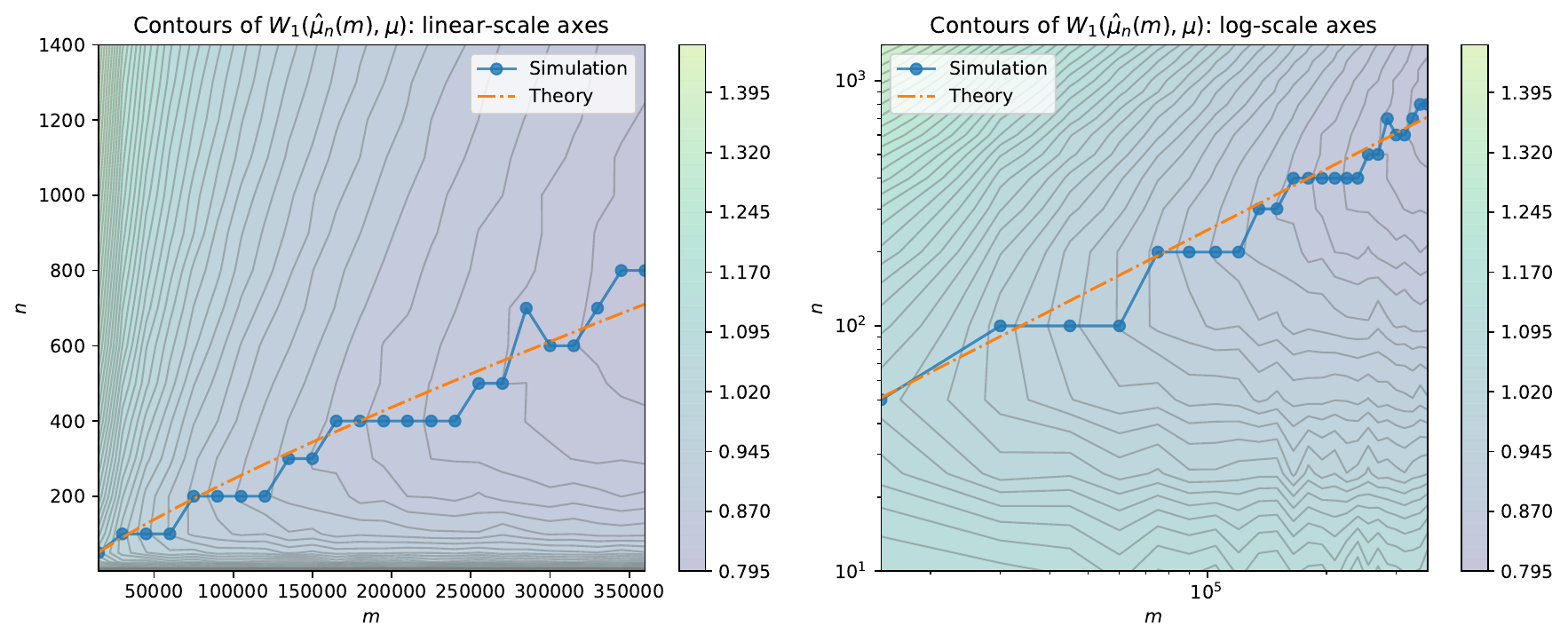}
    \caption{\textbf{Optimal allocation with non-uniform cell weights: Scenario 1 ($U$ dependent on $X$)}. The contours and the colour scale show $W_{1}(\widehat{\mu}_{n}(m), \mu)$ averaged over ten (10) independent realizations of $\widehat{\mu}_{n}(m)$, for various pairs of $m$ and $n$. The $m$ and $n$ axes are plotted in a linear scale on the left panel, and in a log scale on the right panel. The optimal values of $n^*$ for each $m$, as observed in this simulation, are plotted with blue solid dots. Our theoretically derived optimal allocation with estimated intrinsic dimension $k = 9.718$, is plotted as an orange dash-dot line.}
    \label{fig: scenario-1-dep}
\end{figure}

To simulate sequencing with general, non-uniform cell weights, we construct the ground truth distribution $\mu_{(X,U)}$ from real data as before. Since the number of reads that a cell receives is an indication of its sampling frequency, we can use the row sums of $\mathbf{C}$ to help construct $\mu_{(X,U)}$. In order to capture the possible dependence or independence between $X$ and $U$, two scenarios are considered in our experiments.
\paragraph*{\textbf{Scenario 1}} $U$ dependent on $X$. Here, we always match a sampled cell to its original weight in the ground truth counts matrix $\mathbf{C}$. More precisely,
\begin{equation*}
\mu_{(X,U)} = \frac{1}{N} \sum_{\ell=1}^N \delta_{(C_{\ell} /\|C_{\ell}\|_{1}\,,\,\|C_{\ell}\|_{1})}.
\end{equation*}
\paragraph*{\textbf{Scenario 2}} $U$ independent of $X$. Here, a sampled cell is paired with an independently chosen weight from the ground truth counts matrix $\mathbf{C}$:
\begin{equation*}
\mu_{(X,U)} = \frac{1}{N^2} \sum_{\ell,k = 1}^N \delta_{(C_{\ell} /\|C_{\ell}\|_{1}\,, \|C_{k}\|_{1})}.
\end{equation*}

Note that the two scenarios only differ in their couplings of $X$ and $U$; the marginals of $X$ and $U$ remain the same. The $X$-marginal in both cases equals $\mu$ as defined in \eqref{eq: constructed-ground-truth-mu}.

Our optimal allocation theory is agnostic to the possible dependence between $X$ and $U$, and we verify that this is also the case in simulations. We repeat the experiment from \Cref{loc:results_on_the_wot_dataset} under scenarios 1 and 2, and present the results in \Cref{fig: scenario-1-dep} and \Cref{fig: scenario-2-indep} respectively.
The behavior is similar to that seen in \Cref{fig: wot unif size factors}; the observed $m$-$n^*$ relationships are quite close to the theoretically derived optimal allocations~\eqref{eq: optimal n with C}, with $C=4/3$.

\begin{figure}[t]
    \centering
    \includegraphics[width=\textwidth]{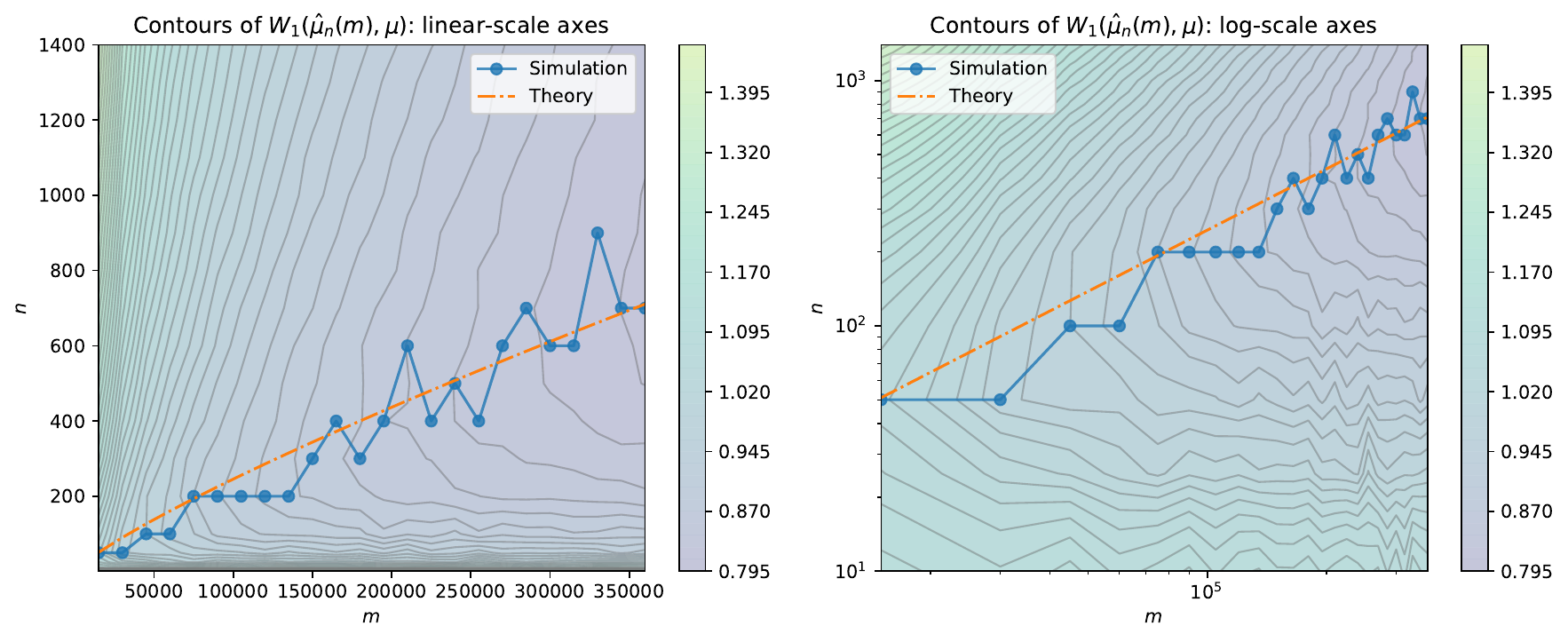}
    \caption{\textbf{Optimal allocation with non-uniform cell weights: Scenario 2 ($U$ independent of $X$)}. The contours and the colour scale show $W_{1}(\widehat{\mu}_{n}(m), \mu)$ averaged over ten (10) independent realizations of $\widehat{\mu}_{n}(m)$, for various pairs of $m$ and $n$. The $m$ and $n$ axes are plotted in a linear scale on the left panel, and in a log scale on the right panel. The optimal values of $n^*$ for each $m$, as observed in this simulation, are plotted with blue solid dots. Our theoretically derived optimal allocation with estimated intrinsic dimension $k = 9.718$, is plotted as an orange dash-dot line.}
    \label{fig: scenario-2-indep}
\end{figure}

\section{Proofs}\label{sec:proofs}
Let's recall the problem setting. $\vec{X}:=\{X_1, \dots, X_n\} \subseteq \Delta_{(d-1)}$ and $\vec{u}=(u_1, \dots, u_n) \in \Delta_{(n-1)}$ denote the set of gene expressions of $n$-cells and the vector of the weight vector of cells, respectively, and $\mu$ is the unknown ground truth distribution. A true empirical distribution and the noisy one are denoted by
\[
    \mu_n := \sum_{i=1}^n \frac{1}{n} \delta_{X_i} \text{ and }   \widehat{\mu}_n(m):= \sum_{i=1}^n \frac{1}{n} \delta_{\widehat{X}_i},
\]
respectively.

\subsection{Proof of \Cref{thm: main result 1}}\label{subsection: thm: main result 1}
The first objective is to control $W_p( \widehat{\mu}_n(m), \mu_n)$. The following lemma allows to bound $W_p( \widehat{\mu}_n(m), \mu_n)$ by $\left( \sum \dist( \widehat{X}_i, X_i)^p \right)^{\frac{1}{p}}$.

\begin{lemma}\cite[Theorem 4.8]{Oldandnew}\label{lemma: convexity of W_p}
Let $\mathcal{X}$ and $\mathcal{Y}$ be Polish spaces, and $c : \mathcal{X} \times \mathcal{Y} \longrightarrow \mathbb{R} \cup \{\infty\}$ be a lower semicontinuous function. For each $\mu \in \mathcal{P}(\mathcal{X})$ and $\nu \in \mathcal{P}(\mathcal{Y})$
\[
    C(\mu, \nu) := \inf_{\pi \in \Pi(\mu, \nu)}\int_{\mathcal{X} \times \mathcal{Y}} c(x,y) d\pi(x,y),
\]
the associated optimal transport cost functional. Let $(\Theta, \lambda)$ be a probability space, and let $\mu_\theta, \nu_\theta$ be two measurable probability measures on $\mathcal{P}(\mathcal{X})$ and $\mathcal{P}(\mathcal{Y})$, respectively. Assume that $c(x,y) \geq a(x) + b(y)$, where $a \in L^1(d\mu_\theta d\lambda(\theta) )$ and $b \in L^1(d\nu_\theta d\lambda(\theta) )$. Then,
\[
    C\left( \int_{\Theta} \mu_\theta d\lambda(\theta),\int_{\Theta} \nu_\theta d\lambda(\theta) \right) \leq \int_{\Theta} C\left(\mu_\theta, \nu_\theta \right) d\lambda(\theta).
\]
\end{lemma}

\Cref{lemma: convexity of W_p} states the convexity of optimal transport functional. Since norms over $\mathbb{R}^{d}$ are lower semicontinuous, Jensen's inequality implies that $W_p$ for any $p \geq 1$ enjoys convexity. More precisely,
\begin{align*}
     W^p_p( \widehat{\mu}_n(m), \mu_n) \leq \sum_{i=1}^n \frac{1}{n} W^p_p(\delta_{\widehat{X}_i}, \delta_{X_i})  = \sum_{i=1}^n \frac{1}{n} \dist(\widehat{X}_i, X_i )^p.
\end{align*}
Hence, one gets
\begin{equation}\label{eq: upper bound of W_p between noisy and empirical}
    W_p( \widehat{\mu}_n(m), \mu_n) \leq  \left( \sum_{i=1}^n \frac{1}{n} \dist( \widehat{X}_i, X_i)^p \right)^{\frac{1}{p}}.
\end{equation}
Therefore, the problem reduces to achieving $\sum \frac{1}{n} \dist( \widehat{X}_i, X_i)^p$ to be small with high probability.

$\sum \frac{1}{n} \dist( \widehat{X}_i, X_i)^p$ would be small by concentration phenomenon if $n$ is large, and also $m$ is large enough proportional to $n$. The concentration inequality of the empirical estimator of the event probability vector of multinomial distribution is well-known. The next lemma is regarding the concentration inequality of it. Originally, \citet[Theorem 2.1]{weissman2003inequalities} concern only $\ell_1$ norm. However, since $\Arrowvert X - X' \Arrowvert_1 \geq \Arrowvert X - X' \Arrowvert_q$ holds on $\Delta_{(d-1)}$ for any $1 \leq q \leq \infty$, one can extend the concentration inequality to $\ell_q$ norms. Recalling the $\ell_0$ norm defined in \eqref{eq: ell_0 norm}, the concentration inequality depends on an error parameter $\delta$ and $\Arrowvert X \Arrowvert_0$ and the number of iteration $m$.

\begin{lemma}\cite[Theorem 2.1]{weissman2003inequalities}\cite[Proposition 1]{qian2020concentration} 
Let $X \in \Delta_{(d-1)}$ and $\widehat{X}=\widehat{X}(m)$ be the empirical estimator of $X$ with $m$ trials. Then, for any $\delta \in [0, 1]$ and $p \in [1, \infty)$,
\begin{equation}\label{eq: concentration of multinomial}
    \mathbb{P}\left( \dist(\widehat{X}, X)^p \leq \left( \frac{2 \Arrowvert X \Arrowvert_0 \log \left( \frac{1}{\delta} \right)}{m} \right)^{\frac{p}{2}} \right) \geq 1 - \delta.
\end{equation}   
\end{lemma}

Let $\epsilon \in (0,1)$ be the fixed error tolerance. Recall \Cref{subsec: setup} that 
\begin{align}\label{eq: def of vec{m}}
    \vec{m} \sim \text{Multinomial}(\vec{u}, m)   
\end{align}
where $\vec{m}:=(m_1, \dots, m_n)$ is a multinomial vector over $n$ integer with $m$ trials and an event probability vector $\vec{u}$. Restricting to a single $i$, each marginal $m_i$ follows binomial distribution with $m$ trials and a success probability $u_i$.

Let $\alpha_i := \exp \left( -\frac{m_i  \epsilon^2}{2 \Arrowvert X_i \Arrowvert_0} \right)$. \eqref{eq: concentration of multinomial} leads to for each $i$ and $k \in \{1, \dots, \lfloor \frac{2^p}{\epsilon^p} \rfloor \}$
\[
    \mathbb{P}\left( \dist(\widehat{X}_i, X_i)^p \geq k\epsilon^p \,|\, \vec{m}, \vec{X}, \vec{u} \right) \leq \alpha_i^{k^{\frac{2}{p}}}.
\]
In particular,
\[
    \mathbb{P}\left( (k-1)\epsilon^p \leq \dist(\widehat{X}_i, X_i)^p < k\epsilon^p \,|\, \vec{m}, \vec{X}, \vec{u} \right) \leq \alpha_i^{(k-1)^{\frac{2}{p}}} - \alpha_i^{k^{\frac{2}{p}}}.
\]
Here $\alpha_i$ is a random variable depending on $m_i$, $\vec{X}$ and $\vec{u}$. Notice that $\dist(\widehat{X}_i, X_i)^p \leq 2^p$ almost surely.

A next lemma is regarding the construction of a new analyzable random variable $Z_i$ which stochastically dominates $\dist(\widehat{X}_i, X_i)^p$. Stochastic dominance is frequently used in various fields including finance, economics, statistics and etc. in order to measure uncertainty among several random variables(vectors). Among many variants of it, here we rely on the first order stochastic dominance.

\begin{definition}\cite[Definition 4.2.1]{roch2024modern} Let $\mu$ and $\nu$ be probability measures on $\mathbb{R}$. The measure $\mu$ is said to stochastically dominate $\nu$ if for all $x \in \mathbb{R}$,
\[
    \mu \left( (x, \infty) \right) \geq \nu \left( (x, \infty) \right).
\]
A real random variable $X$ stochastically dominates $Y$ if the law of $X$ dominates the law of $Y$.    
\end{definition}

\begin{lemma}\label{lem : stochastic dominance}
For each $1\leq i \leq n$ define a random variable $Z_i$ which takes a value $ \{\epsilon^p, \dots, \lfloor \frac{2^p}{\epsilon^p} \rfloor \epsilon^p, 2^p \}$ as
\begin{equation}\label{eq: definition of X_i}
    \mathbb{P}\left( Z_i = k\epsilon^p \,|\, \vec{m}, \vec{X}, \vec{u} \right)   =
    \begin{cases}
    \alpha_i^{(k-1)^{\frac{2}{p}}} - \alpha_i^{k^{\frac{2}{p}}} & \text{if } k = 1, \dots,\lfloor \frac{2^p}{\epsilon^p} \rfloor,\\
    \alpha_i^{\lfloor \frac{2^p}{\epsilon^p} \rfloor^{\frac{2}{p}}} & \text{if } k = \frac{2^p}{\epsilon^p}.
    \end{cases}
\end{equation}
Then $Z_i$ stochastically dominates $\dist(\widehat{X}_i, X_i)^p$ conditioned on $\vec{m}, \vec{X}, \vec{u}$. 
\end{lemma}

\begin{proof}
One needs to show that for any $t \in [0, 2^p]$
\[
    \mathbb{P}\left(Z_i \geq t \, | \, \vec{m}, \vec{X}, \vec{u} \right) \geq \mathbb{P}\left( \dist(\widehat{X}_i, X_i)^p \geq t \,|\, \vec{m}, \vec{X}, \vec{u} \right) 
\]
and the strict inequality holds for some $t$ unless the distribution of $\dist(\widehat{X}_i, X_i)^p$ is equal to that of $Z_i$ conditioned on $\vec{m}, \vec{X}, \vec{u}$. From \eqref{eq: concentration of multinomial} and \eqref{eq: definition of X_i}, for $t \in ((k-1)\epsilon^p, k \epsilon^p )$ it is straightforward that
\begin{align*}
    \mathbb{P}\left( \dist(\widehat{X}_i, X_i)^p \geq t \,|\, \vec{m}, \vec{X}, \vec{u} \right) &\leq \mathbb{P}\left( \dist(\widehat{X}_i, X_i)^p \geq (k-1)\epsilon^p \,|\, \vec{m}, \vec{X}, \vec{u} \right)\\
    &\leq \alpha_i^{(k-1)^{\frac{2}{p}}}\\
    &= \mathbb{P}\left( Z_i \geq (k-1)\epsilon^p \,|\, \vec{m}, \vec{X}, \vec{u} \right).
\end{align*}
\end{proof}

It is immediate that $\sum Z_i$ stochastically dominates $\sum \dist(\widehat{X}_i, X_i)^p$ by the following lemma.

\begin{lemma}\cite[Corollary 4.2.7]{roch2024modern}\label{lem : stocahstic dominance for summation}
If $X_i$ stochastically dominates $Y_i$ for $i=1,2$, then $X_1 + X_2$ stochastically dominates $Y_1 + Y_2$.  
\end{lemma}

One advantage of the construction of $Z_i$'s is that it is easy to compute an upper bound of the expectation of $Z_i$.

\begin{lemma}
Let $Z_{i}$ be the random variable defined in \Cref{lem : stochastic dominance} for fixed $p \in [1,2]$. Then,
\begin{equation}\label{eq: upper bound of conditional Z_i}
    \mathbb{E}\,[Z_{i} \mid \vec{X}, \vec{u}] \leq 
    \begin{cases}
    \left( \epsilon^{p-2} \frac{32 \Arrowvert X_i \Arrowvert}{u_i m} \right) \wedge 2^p &\text{ if $m < \frac{16 \|X_{i}\|_{0}}{u_{i}\epsilon^{2}}$,}\\
    \left( \frac{\epsilon^p}{1 - e^{-2}} \right) \wedge 2^p &\text{ if $m \geq \frac{16 \|X_{i}\|_{0}}{u_{i}\epsilon^{2}}$}
    \end{cases}
\end{equation}
Here, $x \wedge y := \min \{ x,y \}$. 
\end{lemma}

\begin{proof}
    Recall $\alpha_i := \exp \left( -\frac{m_i  \epsilon^2}{2 \Arrowvert X_i \Arrowvert_0} \right)$. A straightforward calculation yields
\begin{align*}
\mathbb{E} \left[ Z_i \,|\,\vec{m}, \vec{X}, \vec{u} \right] &= \sum_{k=1}^{\lfloor \frac{2^p}{\epsilon^p} \rfloor} k\epsilon^p \left( \alpha_i^{(k-1)^{\frac{2}{p}}} - \alpha_i^{k^{\frac{2}{p}}} \right) + 2^p \alpha_i^{\lfloor \frac{2^p}{\epsilon^p} \rfloor^{\frac{2}{p}}}\\
    &= \epsilon^p - \epsilon^p \alpha_i + 2\epsilon^p \alpha_i - 2\epsilon^p \alpha_i^{2^{\frac{2}{p}}} + 3 \epsilon^p \alpha_i^{2^{\frac{2}{p}}} - 3 \epsilon^p \alpha_i^{3^\frac{2}{p}} + \dots + 2^p \alpha_i^{\lfloor \frac{2^p}{\epsilon^p} \rfloor^{\frac{2}{p}}}\\
    &\leq \epsilon^p \sum_{k=0}^{\lfloor \frac{2^p}{\epsilon^p} \rfloor}  \alpha_i^{k^{\frac{2}{p}}}.
\end{align*}
Let $\gamma_i := \exp \left( -\frac{\epsilon^2}{2 \Arrowvert X_i \Arrowvert_0} \right)$ for simplicity. Since each $m_{i}$ is a binomial random variable with parameters $m$ and $u_{i}$, it follows from the moment generating function of $m_{i}$ that
\begin{equation*}
\begin{aligned}
\mathbb{E}[\alpha_{i}^{k^{\frac{2}{p}}}\mid \vec{X}, \vec{u}] = \mathbb{E}[(\gamma_{i}^{k^\frac{2}{p}})^{m_{i}} \mid \vec{X}, \vec{u}] &= ((1-u_{i})+u_{i}\gamma_{i}^{k^\frac{2}{p}})^{m} \\
&= [(1-(1-\gamma_{i}^{k^\frac{2}{p}})u_{i})]^m \leq \exp(-u_{i}m(1-\gamma_{i}^{k^\frac{2}{p}})),
\end{aligned}
\end{equation*}
where we used the fact that $(1-\gamma_{i}^{k^\frac{2}{p}})u_{i} \leq 1$ along with the elementary inequality $1-x \leq \exp(-x)$ for $x \in \mathbb{R}$. 

We study the term $\gamma_{i}^{k^\frac{2}{p}}=\exp\left( -\frac{k^{\frac{2}{p}}\epsilon^2}{2\|X_{i}\|_{0}} \right)$ further. Noting $k \leq \frac{2^p}{\epsilon^p}$ in the sum above, we have 
\begin{equation*}
\frac{k^{\frac{2}{p}}\epsilon^2}{2\|X_{i}\|_{0}} \leq \frac{2^2}{\epsilon^2} \frac{\epsilon^2}{2\|X_{i}\|_{0}} = \frac{2}{\|X_{i}\|_{0}} \leq 2,
\end{equation*}
since $\|X\|_{0} \geq 1$ for any $X \in \Delta_{(d-1)}$. By convexity of $\exp(\cdot)$, we have that for any $x \in[0,2]$,
\begin{equation*}\tag{$*$}
    \exp(-x) \leq 1- \frac{1-e^{-2}}{2}x \leq 1-\frac{x}{4}.
\end{equation*}
As a result, we obtain the bound
\begin{equation*}
1-\gamma_{i}^{k^\frac{2}{p}} = 1-\exp\left( -\frac{k^{\frac{2}{p}}\epsilon^2}{2\|X_{i}\|_{0}} \right) \geq \frac{k^{\frac{2}{p}}\epsilon^2}{8\|X_{i}\|_{0}}
\end{equation*}
which finally gives
\begin{equation*}
\mathbb{E}\left[ \alpha_{i}^{k^{\frac{2}{p}}}\mid \vec{X}, \vec{u} \right] \leq \exp\left( -u_{i}m\left( 1-\gamma_{i}^{k^\frac{2}{p}} \right) \right) \leq \exp \left( - \frac{u_{i}m\epsilon^{2}}{8 \|X_{i}\|_{0}}\,k^{\frac{2}{p}} \right) .
\end{equation*}

Thus, an upper bound of the conditional expectation of $Z_{i}$ is derived as
\begin{equation*}
\begin{aligned}
\mathbb{E}\,[Z_{i} \mid \vec{X}, \vec{u}] &\leq \epsilon^p \sum_{k=0}^{\lfloor \frac{2^p}{\epsilon^p} \rfloor} \mathbb{E}\left[ \alpha_{i}^{k^{\frac{2}{p}}}\mid \vec{X}, \vec{u} \right] & \\
&\leq \epsilon^p \sum_{k=0}^{\lfloor \frac{2^p}{\epsilon^p} \rfloor} \exp \left( - \frac{u_{i}m\epsilon^{2}}{8 \|X_{i}\|_{0}}\,k^{\frac{2}{p}} \right) & \\
&\leq \epsilon^p \sum_{k=0}^{\lfloor \frac{2^p}{\epsilon^p} \rfloor} \exp \left( - \frac{u_{i}m\epsilon^{2}}{8 \|X_{i}\|_{0}}k \right) & (\text{because } p \in[1,2])\\
&\leq \epsilon^p \sum_{k=0}^{\infty} \exp \left( - \frac{u_{i}m\epsilon^{2}}{8 \|X_{i}\|_{0}}k \right) & \\
&\leq \frac{\epsilon^p}{1-\exp \left( - \frac{u_{i}m\epsilon^{2}}{8 \|X_{i}\|_{0}} \right)}.
\end{aligned}
\end{equation*}
If $m < \frac{16 \|X_{i}\|_{0}}{u_{i}\epsilon^{2}}$, i.e., $\frac{u_{i}m\epsilon^{2}}{8 \|X_{i}\|_{0}} \leq 2$, by applying ($*$) again to the last line, one gets
\[
    \frac{\epsilon^p}{1-\exp \left( - \frac{u_{i}m\epsilon^{2}}{8 \|X_{i}\|_{0}} \right)} \leq \epsilon^{p-2} \frac{32 \Arrowvert X_i \Arrowvert}{u_i m}.
\]
Otherwise, if $m \geq \frac{16 \|X_{i}\|_{0}}{u_{i}\epsilon^{2}}$, since $1-\exp \left( - \frac{u_{i}m\epsilon^{2}}{8 \|X_{i}\|_{0}} \right) \geq 1 - e^{-2}$, one gets
\[
    \frac{\epsilon^p}{1-\exp \left( - \frac{u_{i}m\epsilon^{2}}{8 \|X_{i}\|_{0}} \right)} \leq \frac{\epsilon^p}{(1 - e^{-2})}.
\]
The fact that $Z_{i} \leq 2^p$, on the other hand, trivially gives $\mathbb{E}\,[Z_{i} \mid \vec{X}, \vec{u}] \leq 2^p$.
\end{proof}

Combining Lemmas \ref{lem : stochastic dominance} and \ref{lem : stocahstic dominance for summation}, we get
\begin{equation}\label{eq: stochastic dominance}
    \mathbb{P}\left( \frac{1}{n}\sum_{i=1}^n \dist(\widehat{X}_i, X_i )^p \geq t \, \bigg\vert \, \vec{m}, \vec{X}, \vec{u} \right) \leq \mathbb{P}\left( \frac{1}{n} \sum_{i=1}^n Z_i \geq t \, \bigg\vert \, \vec{m}, \vec{X}, \vec{u} \right).
\end{equation}

Although the expectations of $\dist(\widehat{X}_i, X_i)$'s are complicated, we now know an upper bound of the expectations of $Z_i$'s. Hence, the strategy that we adopt is to derive the concentration inequality of the sum of $Z_i$'s by using \eqref{eq: upper bound of conditional Z_i}. Then combined with \eqref{eq: stochastic dominance}, the concentration will give a lower bound of $m$ given $n$ to achieve the error rate at most $\epsilon$.

Recalling \eqref{eq: upper bound of conditional Z_i}, notice that if $m \geq \frac{16 \sum_{i=1}^n \|X_{i}\|_{0}}{ c_* \epsilon^{2}}$,
\[
    \frac{32 \epsilon^{p-2}}{c_* m }\sum_{u_i : u_i \geq \frac{c_*}{n}} \Arrowvert X_i \Arrowvert_0 \leq 2 \epsilon^p.
\]
Hence paying a multiplicative factor at most $2$, from \Cref{assumption} and \eqref{eq: upper bound of conditional Z_i}, an upper bound of the conditional expectation of the sum of $Z_i$'s is derived as
\begin{equation}\label{eq: upper bound of expectation1}
    \mathbb{E} \left[ \frac{1}{n} \sum_{i=1}^n Z_i \, \bigg|\, \vec{X}, \vec{u} \right] \leq \left( \frac{32 \epsilon^{p-2}}{c_* m }\sum_{u_i : u_i \geq \frac{c_*}{n}} \Arrowvert X_i \Arrowvert_0 \right) \wedge 2^p + o \left( \sqrt{\frac{\log n}{n} } \right)
\end{equation}
for all $m \geq 0$. Thus, for any fixed $\epsilon \in (0,1)$, \eqref{eq: upper bound of expectation1} turns out to be 
\[
    \mathbb{E} \left[ \frac{1}{n} \sum_{i=1}^n Z_i \, \bigg|\, \vec{X}, \vec{u} \right] \leq \epsilon^p + o \left( \sqrt{\frac{\log n}{n} } \right) 
\]
if $m$ satisfies
\[
    m \geq \frac{32 \sum_{i=1}^n \|X_{i}\|_{0}}{c_* \epsilon^{2}}.
\]
To obtain a deterministic lower bound of $m$, we take advantage of the fact that $\sum_{i=1}^n  \Arrowvert X_i \Arrowvert_0$ also exhibits a concentration phenomenon.

\begin{lemma}
Fix $\alpha \in (0,1)$. Then,
\begin{equation}\label{eq: concentration of ell^0 norm}
    \mathbb{P}\left( \sum_{i=1}^n \Arrowvert X_i \Arrowvert_0 \geq ( 1+ \alpha) n \mathbb{E}\Arrowvert X \Arrowvert_0 \right) \leq \exp \left( -\frac{2 \alpha^2  (\mathbb{E}\Arrowvert X \Arrowvert_0)^2 n}{\esssup \Arrowvert X \Arrowvert_0} \right).
\end{equation}   
\end{lemma}

\begin{proof}
Notice that each $\Arrowvert X_i \Arrowvert_0$ is bound from $0$ to $d$. Let $c_j$ be the maximum value of the $j$-th entry of $X$. Since $c_j \leq 1$, \eqref{eq: ell_0 norm} leads to $\mu$-almost surely
\[
    \sum_{j=1}^{d} c_j^2 \leq \sup \left\{ \sum_{j=1}^{d} \mathds{1}_{x_j \neq 0} : \mu \big( \{ X \in \Delta_{(d-1)} :\sum_{j=1}^{d} \mathds{1}_{x_j \neq 0} > 0 \} \big) > 0 \right\} = \esssup \Arrowvert X \Arrowvert_0.
\]   
Using Hoeffding's inequality or McDiarmid's inequality straightforwardly with the above bound implies that
\[
    \mathbb{P}\left( \sum_{i=1}^n \Arrowvert X_i \Arrowvert_0 - \mathbb{E}\sum_{i=1}^n \Arrowvert X_i \Arrowvert_0  \geq t \right) \leq \exp \left( -\frac{2 t^2}{n \esssup \Arrowvert X \Arrowvert_0} \right).
\]
Fix $\alpha \in (0,1)$. Choosing $t = \alpha \mathbb{E}\sum_{i=1}^n \Arrowvert X_i \Arrowvert_0 = \alpha n \mathbb{E}\Arrowvert X \Arrowvert_0$, it follows that
\[
    \mathbb{P}\left( \sum_{i=1}^n \Arrowvert X_i \Arrowvert_0 \geq ( 1+ \alpha) n \mathbb{E}\Arrowvert X \Arrowvert_0 \right) \leq \exp \left( -\frac{2 \alpha^2  (\mathbb{E}\Arrowvert X \Arrowvert_0)^2 n}{\esssup \Arrowvert X \Arrowvert_0} \right).
\]
\end{proof}

\begin{remark}
Since $\frac{(\mathbb{E}\Arrowvert X \Arrowvert_0)^2}{ \esssup \Arrowvert X \Arrowvert_0}$ does not depend on $n$, the right hand side of \eqref{eq: concentration of ell^0 norm} decays as $n \to \infty$.    
\end{remark}

Now, we are ready to prove \Cref{thm: main result 1}.


\begin{proof}[Proof of \Cref{thm: main result 1}]
A direct consequence of \eqref{eq: upper bound of expectation1} is that under \Cref{assumption},
\[
    \mathbb{E} \left[ \frac{1}{n} \sum_{i=1}^n Z_i  \right] \leq \left(\frac{32 \epsilon^{p-2} n \mathbb{E}\norm{X}_0}{c_* m} \right) \wedge 2^p + o \left( \sqrt{\frac{\log n}{n} } \right).
\]
Setting $\epsilon^2 = \frac{32 n \mathbb{E}\norm{X}_0}{c_* m}$, and noting that 
\begin{align*}
    \mathbb{E} W_p (\widehat{\mu}_n(m), \mu_n) \leq \left(  \mathbb{E} W_p^p (\widehat{\mu}_n(m), \mu_n)\right)^{\frac{1}{p}} \leq \left( \mathbb{E} \left[ \frac{1}{n} \sum_{i=1}^n Z_i  \right]\right)^{\frac{1}{p}}
\end{align*}
concludes the bound in expectation \eqref{eq: expected W_p between noisy and true empirical dists}.

If $m$ satisfies \eqref{eq: lower bound of m}, combining \eqref{eq: upper bound of expectation1} with \eqref{eq: concentration of ell^0 norm}, it follows that
\[
    \mathbb{E} \left[ \frac{1}{n} \sum_{i=1}^n Z_i \, \bigg|\,  \vec{X}, \vec{u} \right] \leq \epsilon^p - \sqrt{ \frac{\alpha \log n}{n} } + o \left( \sqrt{\frac{\log n}{n} } \right).
\]
Notice that this upper bound is independent of $\vec{X}$ and $\vec{u}$. Therefore, if $m$ satisfies \eqref{eq: lower bound of m}, then
\begin{equation}\label{eq: upper bound of expectation2}
    \mathbb{E} \left[ \frac{1}{n} \sum_{i=1}^n Z_i   \right] = \mathbb{E} \left[ \mathbb{E} \left[ \frac{1}{n} \sum_{i=1}^n Z_i \, \bigg|\,  \vec{X}, \vec{u} \right] \right] \leq \epsilon^p - \sqrt{ \frac{\alpha \log n}{n} } + o \left( \sqrt{\frac{\log n}{n} } \right).
\end{equation}

To obtain an upper bound of the tail probability of $\frac{1}{n} \sum Z_i$, since $Z_i \in [0,2^p]$ almost surely, applying Hoeffding's inequality with \eqref{eq: upper bound of expectation2} leads to
\begin{align*}
    &\mathbb{P}\left( \frac{1}{n} \sum_{i=1}^n Z_i \geq \epsilon^p  \right)\\
    &= \mathbb{P}\left( \frac{1}{n} \sum_{i=1}^n Z_i  - \mathbb{E} \left[ \frac{1}{n} \sum_{i=1}^n Z_i  \right] \geq \epsilon^p - \mathbb{E} \left[ \frac{1}{n} \sum_{i=1}^n Z_i  \right] \right)\\
    &\leq \mathbb{P}\left( \frac{1}{n} \sum_{i=1}^n Z_i  - \mathbb{E} \left[ \frac{1}{n} \sum_{i=1}^n Z_i  \right] \geq \sqrt{ \frac{\alpha \log n}{n} } - o \left( \sqrt{\frac{\log n}{n} } \right) \, \Bigg| \, \sum_{i=1}^n \Arrowvert X_i \Arrowvert_0 \leq ( 1+ \alpha) n \mathbb{E}\Arrowvert X \Arrowvert_0 \right)\\
    &\quad \times \mathbb{P} \left( \sum_{i=1}^n \Arrowvert X_i \Arrowvert_0 \leq ( 1+ \alpha) n \mathbb{E}\Arrowvert X \Arrowvert_0 \right) + \mathbb{P} \left( \sum_{i=1}^n \Arrowvert X_i \Arrowvert_0 \geq ( 1+ \alpha) n \mathbb{E}\Arrowvert X \Arrowvert_0 \right)\\
    &\leq \exp \left( - \frac{n \left( \sqrt{ \frac{\alpha \log n}{n} } -  o \left( \sqrt{\frac{\log n}{n} } \right) \right)^2}{2^{2p-1}} \right) \left( 1 - \exp \left( -\frac{2 \alpha^2  (\mathbb{E}\Arrowvert X \Arrowvert_0)^2 n}{\esssup \Arrowvert X \Arrowvert_0} \right) \right)\\
    &\quad + \exp \left( -\frac{2 \alpha^2  (\mathbb{E}\Arrowvert X \Arrowvert_0)^2 n}{\esssup \Arrowvert X \Arrowvert_0} \right)\\
    &\leq O\left( n^{-\frac{\alpha}{2}} \right).
\end{align*}
Combining this with the stochastic domination relation \eqref{eq: stochastic dominance}, it follows that if $m$ satisfies \eqref{eq: lower bound of m},
\begin{align*}
    \mathbb{P}\left( \frac{1}{n}\sum_{i=1}^n \dist(\widehat{X}_i, X_i ) \geq \epsilon^p \right) &= \mathbb{E} \left[  \mathbb{P}\left( \frac{1}{n}\sum_{i=1}^n \dist(\widehat{X}_i, X_i ) \geq \epsilon^p \, \bigg\vert \,  \vec{m}, \vec{X}, \vec{u} \right) \right]\\ 
    &\leq \mathbb{E} \left[  \mathbb{P}\left( \frac{1}{n} \sum_{i=1}^n Z_i \geq \epsilon^p \, \bigg\vert \, \vec{m}, \vec{X}, \vec{u} \right) \right]\\
    &=\mathbb{P}\left( \frac{1}{n} \sum_{i=1}^n Z_i \geq \epsilon^p \right)\\
    &\leq O\left( n^{-\frac{\alpha}{2}} \right).
\end{align*}
Therefore, the conclusion follows from \eqref{eq: upper bound of W_p between noisy and empirical}.
\end{proof}

\subsection{Proof of \Cref{thm: lower bound}}

Since $W_{p}\geq W_{p'}$ for $p\geq p'$ by Jensen's inequality~\cite[Remark 6.6]{Oldandnew}, it suffices to lower bound $W_{1}$. By Kantorovich duality~\cite[Theorem 1.14]{villani2003topics}, we have that
\begin{equation*}
W_{1}(\widehat{\mu}_{n}(m), \mu_{n}) \geq \int f\,d \widehat{\mu}_{n}(m) - \int f\,d \mu_{n},
\end{equation*}
for any function $f:\Delta_{(d-1)} \to \mathbb{R}$ that is 1-Lipschitz with respect to $\mathrm{dist}(\cdot,\cdot)$ on $\Delta_{(d-1)}$. We choose $f(X)= \frac{\|X\|_{2}^2}{2}$, and verify its Lipschitz property: for any $X, X' \in \Delta_{(d-1)}$ and $q \in [1,2]$,
\begin{align*}
    |f(X')-f(X)| &= \frac{1}{2} \left| \|X'\|_2^2 - \|X\|_{2}^2 \right|\\
    &= \frac{1}{2}\left( \|X'\|_2 + \|X\|_{2} \right) \left| \|X'\|_2 - \|X\|_{2} \right| \\
    & \leq \frac{1}{2} \cdot 2 \cdot \|X'-X\|_2 \leq \|X'-X\|_q.
\end{align*}
Now, recall that $\mu_{n} = \frac{1}{n} \sum  \delta_{X_i}$ and $\widehat{\mu}_{n}(m)=\frac{1}{n} \sum  \delta_{\widehat{X}_{i}}$. Furthermore, recall $\vec{m}=(m_{1}, \dots, m_{n})$ from \eqref{eq: def of vec{m}} where we have assumed $\boldsymbol{u}=\left(\frac{1}{n}, \dots ,\frac{1}{n} \right)$. Such fixed $\boldsymbol{u}$ also implies that $\vec{m}$ and $\boldsymbol{X}$ are independent. Putting these definitions together, we can write
\begin{align*}
    &\mathbb{E}W_{1}(\widehat{\mu}_{n}(m), \mu_{n})\\
    &\geq \mathbb{E}\left[ \int f \,d\widehat{\mu}_{n}(m) - \int f\,d\mu_{n} \right] \\
    &= \mathbb{E} \frac{1}{n}\sum_{i=1}^{n} \left(  \frac{1}{2}\|\widehat{X}_{i}\|_{2}^2 - \frac{1}{2}\|X_{i}\|_{2}^2 \right) \\
    &= \frac{1}{n}\sum_{i=1}^{n} \frac{1}{2} \underbrace{\mathbb{E}\mathds{1}_{\{ m_{i}>0 \}} \left( \|\widehat{X}_{i}\|_{2}^2 - \|X_{i}\|_{2}^2 \right)}_{\mathrm{(a)}} + \frac{1}{n}\sum_{i=1}^{n} \frac{1}{2} \underbrace{\mathbb{E}\mathds{1}_{\{ m_{i}=0 \}} \left( \|\widehat{X}_{i}\|_{2}^2 - \|X_{i}\|_{2}^2 \right)}_{\mathrm{(b)}},
\end{align*}
so that the problem simplifies to lower bounding the terms labeled (a) and (b). 

We start with the simpler term (b). As stated in \Cref{subsec: setup}, $\widehat{X}_{i}$ may be chosen arbitrarily in $\Delta_{(d-1)}$ when $m_{i}=0$, and hence, we use the trivial lower bound $\|\widehat{X}_{i}\|_{2}^2 - \|X_{i}\|_{2}^2 \geq -1$. Using $m_{i} \sim \mathrm{Binomial}(m, \frac{1}{n})$, this can be followed up with the estimate
\begin{equation*}
\mathbb{E}\mathds{1}_{\{ m_{i}=0 \}} = \mathbb{P}(m_{i}=0) = \left( 1- \frac{1}{n} \right)^{m} \leq \exp \left(-\frac{m}{n} \right)
\end{equation*}
to conclude that $(\mathrm{b})\geq -\exp \left(-\frac{m}{n} \right)$.

Next, we consider the term (a), rewriting it as
\begin{equation*}
\mathbb{E} \mathds{1}_{\{ m_{i} >0 \}} \left( \|\widehat{X}_{i}\|_{2}^2 - \|X_{i}\|_{2}^2 \right) = \mathbb{E}\left[ \mathds{1}_{\{ m_{i} >0 \}} \,\mathbb{E}\left[  \|\widehat{X}_{i}\|_{2}^2 - \|X_{i}\|_{2}^2  \,\bigm|\, \vec{m}, \boldsymbol{X}, \boldsymbol{u} \right] \right].
\end{equation*}
Using the fact that
\begin{equation*}
\widehat{X}_{i} \mid \vec{m}, \boldsymbol{X}, \boldsymbol{u} \sim \frac{1}{m_{i}} \mathrm{Multinomial}(X_{i}, m_{i})
\end{equation*}
whenever $m_{i}>0$, one can write the moments
\begin{align*}
    \mathbb{E}\left[ \widehat{X}_{i} \bigm| \vec{m}, \boldsymbol{X}, \boldsymbol{u} \right] &= X_{i}, \\
    \mathrm{tr}\,\left(\mathrm{Cov}\left[\widehat{X}_{i} \bigm| \vec{m}, \boldsymbol{X}, \boldsymbol{u} \right] \right) &= \frac{1-\|X_{i}\|_{2}^2}{m_{i}},
\end{align*}
so that
\begin{equation*}
\mathbb{E}\left[ \|\widehat{X}_{i}\|_{2}^2 \bigm| \vec{m}, \boldsymbol{X}, \boldsymbol{u} \right] = \|X_{i}\|_{2}^2 + \frac{1}{m_{i}}(1- \|X_{i}\|_{2}^2).
\end{equation*}
Hence, we simplify
\begin{align*}
    \mathbb{E} \left[ \mathds{1}_{\{ m_{i} >0 \}} \,\mathbb{E}\left[  \|\widehat{X}_{i}\|_{2}^2 - \|X_{i}\|_{2}^2  \bigm| \vec{m}, \boldsymbol{X}, \boldsymbol{u} \right] \right] &= \mathbb{E} \left[ \mathds{1}_{\{ m_{i}>0 \}} \frac{1}{m_{i}}(1 - \|X_{i}\|_{2}^2) \right]\\
    &= \mathbb{E} \left[ \mathds{1}_{\{ m_{i}>0 \}} \frac{1}{m_{i}}  \right] (1 - \mathbb{E}\|X_{i}\|_{2}^2)
\end{align*}
using the independence of $\vec{m}$ and $\boldsymbol{X}$, and then lower bound
\begin{align*}
    \mathbb{E} \left[  \mathds{1}_{\{ m_{i}>0 \}} \frac{1}{m_{i}}   \right] = \mathbb{E}\left[ \frac{1}{m_{i}} \Bigg| m_{i}>0 \right] \,\mathbb{P} (m_{i}>0) \geq \frac{1}{\mathbb{E}\left[ m_{i} \mid m_{i}>0 \right] } \, \mathbb{P}(m_{i}>0).
\end{align*}
Finally, note that
\begin{equation*}
    \mathbb{E}[m_{i} \mid m_{i}>0] = \frac{\mathbb{E}m_{i}}{\mathbb{P}(m_{i}>0)} =\frac{m}{n} \frac{1}{\mathbb{P}(m_{i}>0)},
\end{equation*}
and $\mathbb{P}(m_{i}>0) = 1- \mathbb{P}(m_{i}=0) \geq 1- \exp \left( -\frac{m}{n} \right)$, which concludes the calculation for (a).

Putting everything together, 
\begin{align*}
\mathbb{E}W_{1}(\widehat{\mu}_{n}(m), \mu_{n}) &\geq \frac{1}{2}\left[ \frac{n}{m} \left( 1-\exp \left( -\frac{m}{n} \right) \right)^2 \left(1- \mathbb{E}\|X\|_2^2 \right) - \exp \left( -\frac{m}{n} \right) \right] \\
&\geq \frac{1-\mathbb{E}\|X\|_2^2}{4} \frac{n}{m},
\end{align*}
provided $m \geq 2n \log \frac{4}{1-\mathbb{E}\|X\|_2^2}$.

\subsection{Proof of \Cref{thm: main result 2}}
The final goal is to obtain an optimal number of cells $n$ given the $m$ reads, which determines the number of reads per cell, in order to minimize $W_p(\widehat{\mu}_n(m), \mu)$ for $p \in [1,2]$. We only focus on the optimal order of $m$ in terms of $n$ and ignore constants. Observe that a triangle inequality for $W_p$ shows
\[
    W_p( \widehat{\mu}_n(m), \mu_n) + W_p( \mu_n, \mu).
\]
Since $W_p( \widehat{\mu}_n(m), \mu_n)$ is controlled by \Cref{thm: main result 1}, let's pay attention to $W_p( \mu_n, \mu)$.

There have been many papers studying to deduce the tight rate of the convergence of empirical distributions and concentration inequalities of $W_p$ for $p \geq 1$. Regarding the state-of-art results in this area, we refers readers to \cite{dudley69, MR2280433, MR2861675, MR3189084, NF_AG_rate_Wasserstein, Niles-Weed_Bach, MR4359822, merigot2021non, MR4255123, MR4441130} and references therein. In the present paper we adopt \cite{Niles-Weed_Bach} since $\Delta_{(d-1)}$ is bounded.

\begin{lemma}\cite[Theorem 1, Proposition 20]{Niles-Weed_Bach}\label{lemma: theorem Niles-Weed_Bach}
Let $\mu$ have support with diameter $D$, and let $d^*_p(\mu)$ be the upper Wasserstein dimension. Let $1 \leq p < \infty$. If $k > \max\{ d^*_p(\mu) , 2p\}$, then
\begin{equation}\label{eq: upper bound of convergence rate}
    \mathbb{E}[W_{p}(\mu,\mu_n)] \leq C_1 n^{-\frac{1}{k}} + C_2 n^{-\frac{1}{2p}} = O\left( n^{-\frac{1}{k}} \right)
\end{equation}
for some $C_1$ and $C_2$ which are given as
\[
    C_1 = 3^{\frac{3k}{k - 2p}  + \frac{1}{p}} \left( \frac{1}{ 3^{\frac{k}{2}-p} - 1} + 3 \right)^{\frac{1}{p}}, \quad C_2 = \left( \frac{27}{\varepsilon'}\right)^{\frac{k}{2p}}
\]
where $\varepsilon' > 0$ is a constant which satisfies $d(\varepsilon, \varepsilon^{\frac{sp}{s - 2p}}; \mu) \leq k$ for all $\varepsilon \leq \varepsilon'$.

Furthermore, for any $t > 0$, 
\begin{equation}\label{eq: concentration of convergence rate}
    \mathbb{P}\left(W_{p}^{p}(\mu,\mu_n)\geq t +\mathbb{E}[W_{p}^{p}(\mu,\mu_n)]\right)\\
 \leq \exp\left( - \frac{2nt^2}{D^{2p}} \right).
\end{equation}
\end{lemma}

\begin{remark}
As mentioned in \Cref{rmk: intrinsic dimension}, such $k$ is the intrinsic dimension of $\mu$, or the dimension of its support.

Note that in equation \eqref{eq: upper bound of convergence rate}, the second term (involving $C_2$) is small compared to the first term since $k> 2p$. Therefore, while $C_2$ may depend on $\mu$ (through $\varepsilon'$), the  leading term does not depend on the dimension.

There is an alternative upper bound in \cite{Niles-Weed_Bach}[Proposition 10] to get rid of $C_2$. Let $d_{\geq \varepsilon}(\tau; \mu):= \sup_{\varepsilon' \in [\varepsilon, \frac{1}{9}]} d(\varepsilon, \tau ; \mu)$, and $k_n := \inf_{\varepsilon > 0} \max \left\{ d_{\geq \varepsilon}(\varepsilon^p; \mu), \frac{\log n}{- \log \varepsilon} \right\}$. If $k_n > 2p$, then
\[
    \mathbb{E}[W_p(\mu, \mu_n)] \leq C_1 n^{-\frac{1}{k_n}}
\]
where $C_1$ depends on $p$ and $k_n$. However the choice of $k_n$ is harder to interpret than $d^*_p(\mu)$. Also, it seems to depend on $\mu$ itself.
\end{remark}

\begin{proof}[Proof of \Cref{thm: main result 2}]
Fix $\alpha \in (0,1)$ and $p \in [1,2]$. Assume that there is some $k > 0$ such that $k > \max\{d^*_p(\mu), 2p\}$. Since $\text{diam}(\Delta_{(d-1)}) \leq 2$, by \eqref{eq: concentration of convergence rate}, $W_p( \mu_n, \mu) \leq \mathbb{E} W_p( \mu_n, \mu) + n^{-q}$ with probability $1 - \exp \left( -\frac{1}{2^{2p-1}}n^{1 - 2q} \right)$ for any $0 < q < \frac{1}{2}$. Let $q=\frac{1}{4}$ for simplicity. Then, with probability $1 - \exp \left( -\frac{\sqrt{n}}{2^{2p-1}}  \right)$
\[
    W_p( \mu_n, \mu) \leq O(n^{-\frac{1}{k}}).
\]
Let $\epsilon = n^{-t}$ for some $t > 0$. Given $m$, let $n$ satisfy \eqref{eq: lower bound of m}. Ignoring the lower order term, a maximum such $n$ should satisfy
\begin{equation}\label{eq : fixed m, maximum n with concetanation}
    n \geq \left( \frac{c_*  m}{8 ( 1+ \alpha) \mathbb{E} \Arrowvert X \Arrowvert_0 } \right)^{\frac{1}{1+2t}}.
\end{equation}

Combining \eqref{eq: W_p between noisy and true empirical dists} with \eqref{eq: concentration of convergence rate}, a triangle inequality implies
\[
    W_p( \widehat{\mu}_n(m), \mu) \leq O(n^{-t}) + O(n^{-\frac{1}{k}})
\]
with probability $1 - O\left( n^{-\frac{\alpha}{2}} \right) = \left( 1 - O\left( n^{-\frac{\alpha}{2}} \right) \right) \left( 1 - \exp \left(-\frac{\sqrt{n}}{2^{2p-1}}  \right) \right)$. Let's optimize $O(n^{-t}) + O(n^{-\frac{1}{k}})$ by choosing $t$. Let $n(t)$ satisfy \eqref{eq : fixed m, maximum n with concetanation} given $m$ and $t$. Then, 
\begin{align*}
    O(n(t)^{-t}) + O(n(t)^{-\frac{1}{k}}) &= O \left( e^{-t \log n(t)} \right) +  O \left( e^{-\frac{1}{k} \log n(t)} \right) = O \left( e^{-\min \{t, \frac{1}{k} \} \log n(t)} \right).
\end{align*}
If $t > \frac{1}{k}$, the order is $O\left( n(t)^{-\frac{1}{k}} \right)$. Notice that $n(t)$ is decreasing with respect to $t$: it is trivial due to \eqref{eq : fixed m, maximum n with concetanation}. Hence, $n(t)^{-\frac{1}{k}}$ decreases as $t$ decreases until $t = \frac{1}{k}$. If $t < \frac{1}{k}$, the order is 
\[
    O\left( n(t)^{-t} \right) = O\left( e^{-t \log n(t)} \right) = O\left( \exp \left(- \frac{ t \left(\log (c_* m) - \log\left( 8 ( 1+ \alpha) \mathbb{E} \Arrowvert X \Arrowvert_0 \right) \right) }{1+2t} \right) \right).
\]
It is easy to check that
\begin{align*}
    &\partial_t \left(  \frac{ t \left(\log (c_* m) - \log\left( 8 ( 1+ \alpha) \mathbb{E} \Arrowvert X \Arrowvert_0 \right) \right) }{1+2t} \right)\\
    &= \frac{\left(\log (c_* m) - \log\left( 8 ( 1+ \alpha) \mathbb{E} \Arrowvert X \Arrowvert_0 \right) \right) -2t \left(\log (c_* m) - \log\left( 8 ( 1+ \alpha) \mathbb{E} \Arrowvert X \Arrowvert_0 \right) \right)  }{(1+2t)^2}\\
    &\geq 0
\end{align*}
provided that $\log (c_* m) \geq \log\left( 8 ( 1+ \alpha) \mathbb{E} \Arrowvert X \Arrowvert_0 \right)$ and $t \leq \frac{1}{2}$. Let $m_0 (c_*, \alpha, \mathbb{E} \Arrowvert X \Arrowvert_0 )$ be a positive constant satisfying $\log (c_* m) > \log\left( 8 ( 1+ \alpha) \mathbb{E} \Arrowvert X \Arrowvert_0 \right)$. Since $m \geq m_0$ and $k > 4$, increasing $t$ up to $\frac{1}{k}$ decreases $O\left( n(t)^{-t} \right)$. Combining the case for $t > \frac{1}{k}$, we can conclude that the optimal $t = \frac{1}{k}$.
\end{proof}

\section{Conclusion and future works}\label{sec:conclusion}

In this paper we analyze the distance between an empirical distribution of cells, sampled with scRNA-seq, and the true population. 
When the total number of reads is fixed, sampling more cells does not always make the empirical distribution $\hat \mu$ closer to the true distribution $\mu$. In order to minimize the Wasserstein distance between $\hat \mu$ and $\mu$, the number of cells $n$ should scale with the number of reads $m$ like
\[
    n\sim m^{\frac{k}{k+2}},
\]
where $k$ is the intrinsic dimension of the true distribution.

In future work, it would be interesting to analyze optimal sequencing depth for different estimation problems. Fewer reads per cell may be optimal for structured estimation problems such as clustering or shape-constrained density estimation. And adding additional dimensions such as space or time would present interesting analytical opportunities for spatial transcriptomics or trajectory inference. 
Our results for general empirical distributions should provide a good starting point to analyze these diverse contexts and extract as much information as possible from a limited sequencing budget.

We note that it would be feasible to explore different mathematical models for noisy empirical distributions, the beyond normalized multinomial sampling of reads for scRNA-seq. Indeed, most data are observed with measurement noise, and multinomial sampling plays no crucial role in our analysis: the concentration inequality, \eqref{eq: concentration of multinomial}, is the only property used in the proof.
We hope that our work serves as a starting point for investigating noisy empirical distributions in general settings.

Finally, one could ask the same question under high-dimensional setting. In our model, the ambient and intrinsic dimensions $d$ and $k$ are fixed, and only $n$ and $m$ grow. In various high-dimensional statistics, however, $d$ also grows as $n$ does (also $k$ sometimes grows but slower than $d$). What happens $d$ also grows? In this case, $\frac{m}{n} \to \alpha$ and $\frac{d}{n} \to \beta$. Under this setting, can the Wasserstein distance between $\widehat{\mu}_n(m)$ and $\mu$ vanish? Is there any phase transition of the Wasserstein distance in terms of $\alpha$ and $\beta$? We believe it is a new type of high-dimensional problem of interest both theoretically and practically.

%
%
%
%

\begin{acks}[Acknowledgments]
The majority of the work for this publication was done when JK was at University of British Columbia and PIMS. JK appreciates their supports and hospitality.
\end{acks}

\begin{funding}
JK thanks to PIMS postdoctoral fellowship  through the  Kantorovich Initiative PIMS Research Network (PRN) as well as National Science Foundation grant NSF-DMS 2133244. GS and SK are supported by a CIHR PG and an NSERC DG. We thank PIMS for their generous support; report identifier PIMS-20240912-PRN01.
\end{funding}


\begin{appendix}
\section{Results on additional datasets}
\label{app: results-additional-datasets}
We validate our theoretical results on two other scRNA-seq datasets -- the \emph{sea urchin} dataset~\cite{massri_SingleCellTranscriptomicsReveals_2025} and the \emph{Arabidopsis} dataset~\cite{shahanSinglecellArabidopsisRoot2022}. Further details about the specific data files and preprocessing steps are available in \Cref{appsubsec: dataset-preprocessing}.

As discussed in \Cref{loc:experimental_setup}, the datasets contain cells-by-genes counts matrices $\mathbf{C}$ of (UMI) reads recorded in sequencing experiments. The corresponding population distributions $\mu$ are constructed as per~\eqref{eq: constructed-ground-truth-mu}, with statistics
\begin{itemize}
    \item $d=1000, N = 5580, k=10.092$, and $\mathbb{E}_{P \sim \mu}\|X\|_0 \simeq 113$ for the sea urchin dataset,
    \item $d=1000, N=5801, k=7.432$ and $\mathbb{E}_{P \sim \mu}\|X\|_0 \simeq 90$ for the \textit{Arabidopsis} dataset.
\end{itemize}
Here, the intrinsic dimensions $k$ are approximated via the heuristic from \Cref{loc:assessing_the_intrinsic_dimension}; see \Cref{appsubsec: population-statistics}
 and \Cref{fig:k-estimation-plot}. We repeat the experiment from \Cref{loc:results_on_the_wot_dataset} on the two datasets below; in particular, we assume uniform cell weights.

As in \Cref{loc:results_on_the_wot_dataset}, for various pairs of $n$ and $m$, we simulate the noisy empirical distribution $\widehat{\mu}_{n}(m)$ for ten (10) independent trials, and compute the mean Wasserstein error $W_{1}(\widehat{\mu}_{n}(m), \mu)$ over those trials. Contours of mean $W_{1}(\widehat{\mu}_{n}(m), \mu)$ over $m$-$n$ axes, for both datasets, are plotted in \Cref{fig: urchin-root unif size factors}. For each budget $m$ tested, we identify the number of cells $n=n^*$ that yields the lowest mean $W_{1}(\widehat{\mu}_{n}(m), \mu)$. These $m$-$n^*$ relationships, also shown in \Cref{fig: urchin-root unif size factors}, appear to be well captured by our theoretically derived optimal allocation 
$$
n \sim \left( \frac{Cm}{\mathbb{E}\|X\|_0} \right)^{1 - \frac{2}{2+k}}
$$
where $C$ is a $\Theta(1)$ constant. In the plots, we used $C=1$ for the sea urchin dataset, and $C=6$ for the \textit{Arabidopsis} dataset.

\begin{figure}[h]
\centering
\includegraphics[width=\textwidth]{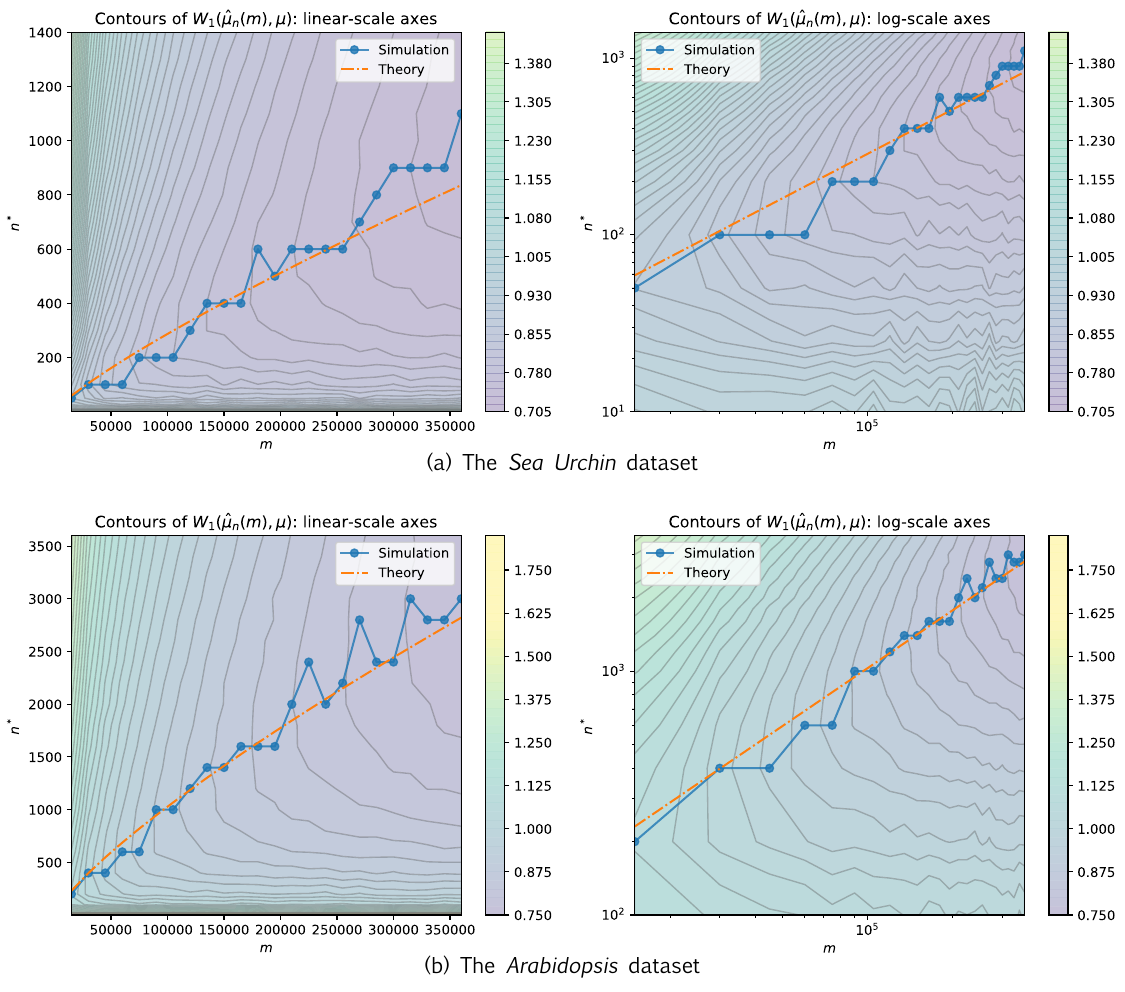}
\caption{\textbf{Optimal allocation ($m$-$n$ relationship) with uniform cell weights for (a) the sea urchin dataset, and (b) the \textit{Arabidopsis} dataset.} For each dataset, the contours and the colour scales show $W_{1}(\widehat{\mu}_{n}(m), \mu)$ averaged over ten (10) independent realizations of $\widehat{\mu}_{n}(m)$, for various pairs of $m$ and $n$ on a grid. The $m$ and $n$ axes are plotted in a linear scale on the left panels, and in a log scale on the right panels. The optimal values of $n^*$ for each $m$, as observed in these simulations, are plotted with blue solid dots. Our theoretically derived optimal allocations are plotted as orange dash-dot lines.
}
\label{fig: urchin-root unif size factors}
\end{figure}

\section{Experiment details}
\label{app: experimental-details}

In this section, we provide further details on the dataset, preprocessing steps, and software packages used in the experiments presented in \Cref{sec:Empiricalresults}.

\subsection{Datasets and preprocessing}
\label{appsubsec: dataset-preprocessing}
Data processing is done in python, using the packages Scanpy~\cite{wolf2018scanpy} and Anndata~\cite{anndata}. Below, we describe details and preprocessing steps specific to the three datasets studied. The resulting preprocessed counts matrices $\mathbf{C}$ are then used as described in \Cref{sec:Empiricalresults}.

\subsubsection*{The cellular reprogramming dataset}
This dataset is published in \citet{schiebingerOptimalTransportAnalysisSingleCell2019}, and is publicly available through the NCBI Gene Expression Omnibus with identification number GSE122662. We use the counts matrices from the files \path{GSM3195748_D15.5_serum_C1_gene_bc_mat.h5} and
\path{GSM3195749_D15.5_serum_C2_gene_bc_mat.h5}. 
Scanpy's function \texttt{read\_10x\_h5} is used to read in the data files, and the resulting Anndata objects are concatenated via an inner join. The result is a (annotated) cells-by-genes counts matrix of unique molecular identifier (UMI) reads. 

We preprocess this data by first filtering out genes that are expressed in fewer than 10 cells, and then subsetting to the top $d=1000$ highly variable genes (as determined by Scanpy's \texttt{highly\_variable\_genes} function with \texttt{flavor = "seurat\_v3"}). Note that we do \textbf{not} use any nonlinear transformation on the data, such as the standard \texttt{log1p} transform.

\subsubsection*{The sea urchin dataset}

This dataset is affiliated with \citet{massri_SingleCellTranscriptomicsReveals_2025}.
The particular data file we used is made available on \url{https://personal.math.ubc.ca/~sharvaj/} as \texttt{adata\_raw\_0715\_SCT.h5ad}, whose data matrix consists of counts post an sctransform~\cite{hafemeisterNormalizationVarianceStabilization2019} followed by a \texttt{log1p} transform. Scanpy's function \texttt{read\_h5ad} is used to read in the data, and the resulting Anndata object is subsampled to 6000 cells at random (using \texttt{scanpy.pp.subsample}) to reduce the computational cost. Next the entries of data matrix (\texttt{AnnData.X}) are transformed as $t \mapsto \exp(t)-1$ in order to reverse the \texttt{log1p} transform and obtain a cells-by-genes matrix of integer counts. 

We preprocess this matrix by filtering out cells with fewer than 2000 counts, and then filtering out genes that are expressed in fewer than 10 cells. Finally, we subset to the top $d = 1000$ highly variable genes (as determined by Scanpy’s \texttt{highly\_variable\_genes} function with \texttt{flavor = "seurat\_v3"}). 

\subsubsection*{The \textit{Arabidopsis} dataset}

The particular data file we used is made available on \url{https://personal.math.ubc.ca/~sharvaj/} as \texttt{recovery\_counts.h5ad}, which contains a cells-by-genes matrix of integer counts. Scanpy's function \texttt{read\_h5ad} is used to read in the data, and the resulting Anndata object is subsampled to 6000 cells at random (using \texttt{scanpy.pp.subsample}) to reduce the computational cost. 

As before, we filter out cells with fewer than 2000 counts, and then filter out genes that are expressed in fewer than 10 cells. Finally, we subset to the top $d = 1000$ highly variable genes (as determined by Scanpy’s \texttt{highly\_variable\_genes} function with \texttt{flavor = "seurat\_v3"}).

\subsection{Estimating intrinsic dimensions}
\label{appsubsec: population-statistics}

Recall that $\mu$ is constructed as
\begin{align}
    \mu = \frac{1}{N} \sum_{\ell=1}^N \delta_{C_{\ell} / \|C_{\ell}\|_{1}}
    \label{eq: app atomic mu}
\end{align}
where $C_\ell$ is the $\ell$-th row of the (preprocessed) counts matrix $\mathbf{C}$. As we keep only the top 1000 highly variable genes, the ambient dimension is $d=1000$ for all the datasets considered.

We employ the heuristic discussed in \Cref{loc:assessing_the_intrinsic_dimension} to estimate the intrinsic dimension $k$ of $\mu$. In particular, for various values of $n$, we repeatedly draw i.i.d. samples from $\mu$ to form $n$-sample (true) empirical distributions $\mu_{n}$, and compute $W_{1}(\mu,\mu_{n})$ for each such trial. Then, we fit a straight line to the resulting set of pairs $(\log n, \log W_{1}(\mu,\mu_{n}))$, the slope of which approximates $-1 /k$. As seen in \Cref{fig:k-estimation-plot}, this procedure gives $k=9.718$ for the reprogramming dataset, $k=10.092$ for the sea urchin dataset, and $k=7.432$ for the Arabidopsis dataset.


\begin{figure}[h]
    \centering
    \includegraphics[width=\textwidth]{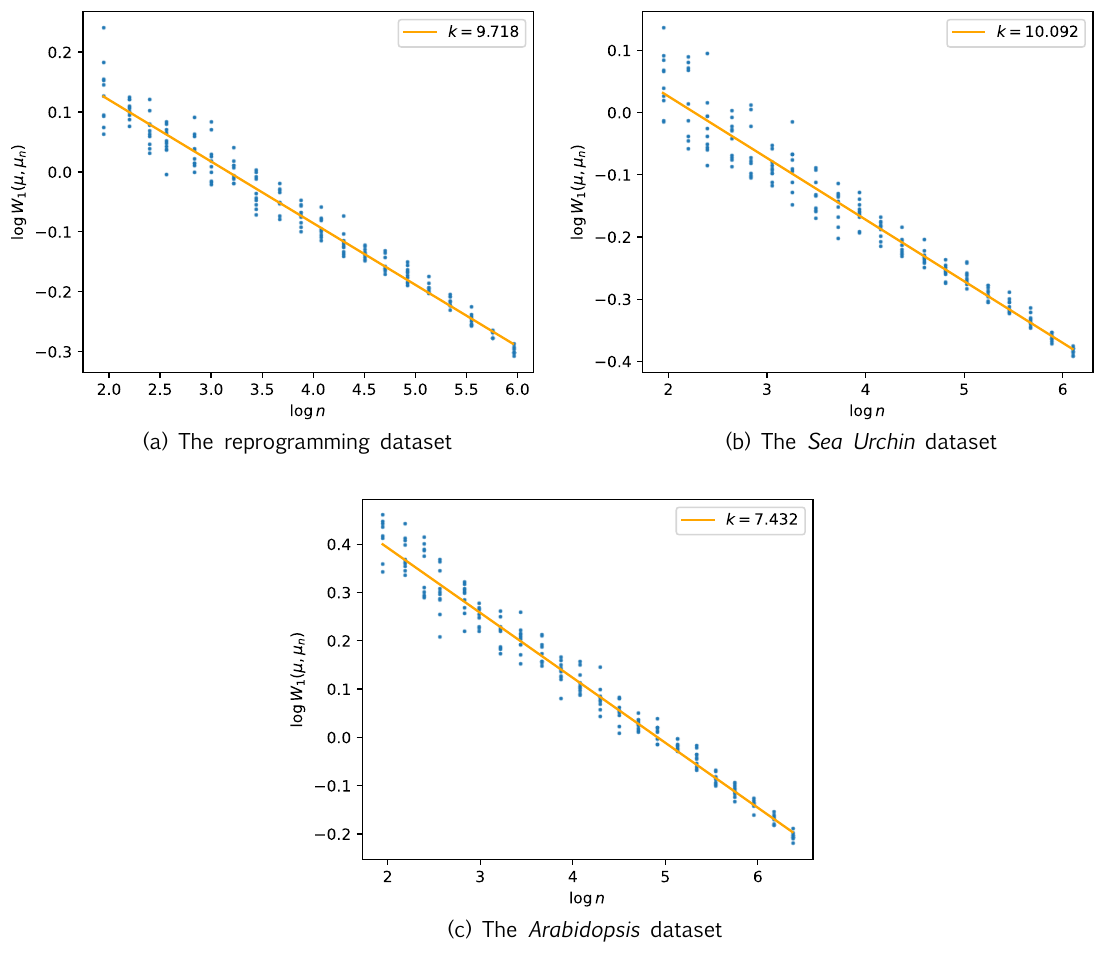}
    \caption{
  \textbf{Estimating the intrinsic dimension $k$ of $\mu$ constructed from (a) the reprogramming dataset, (b) the sea urchin dataset, and (c) the \textit{Arabidopsis} dataset.}  
    As per the heuristic from  \Cref{loc:assessing_the_intrinsic_dimension}, we fit a straight line to the $\log n$-$\log W_{1}(\mu,\mu_{n})$ relationship, the slope of which approximates $-1 /k$.}
    \label{fig:k-estimation-plot}
\end{figure}

\subsection{Constructing lower-dimensional distributions $\mu_k$}\label{appsubsec: lower-dim}

\begin{figure}[h]
    \centering
    \includegraphics[width=\textwidth]{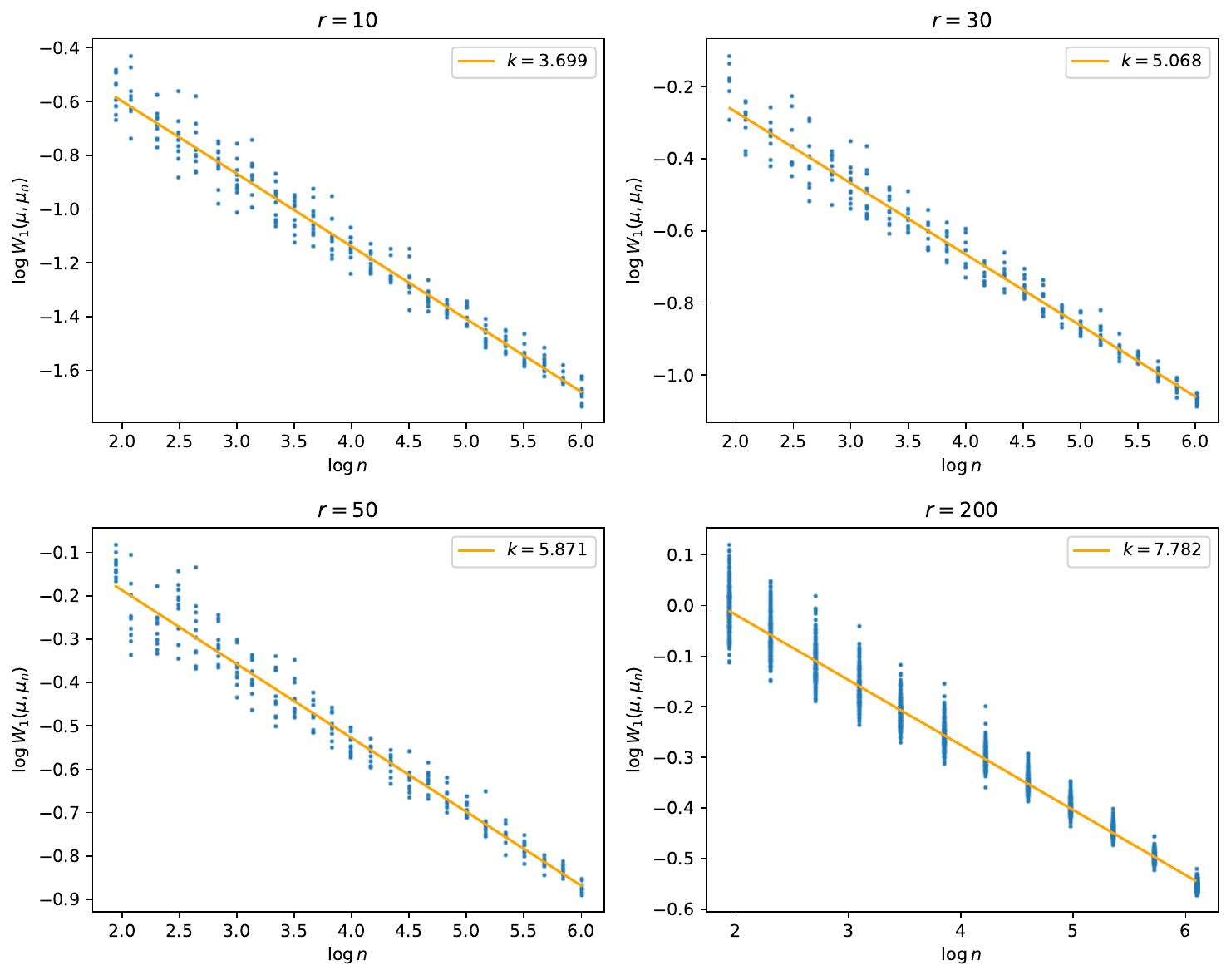}
    \caption{\textbf{Estimating the intrinsic dimensions $k$ of the synthetic distributions $\mu_k$ constructed as in \Cref{appsubsec: lower-dim}.}  The rank $r$ is the one used for matrix factorization. As per the heuristic from  \Cref{loc:assessing_the_intrinsic_dimension}, we fit a straight line to the $\log n$-$\log W_{1}(\mu,\mu_{n})$ relationship, the slope of which approximates $-1 /k$.}
    \label{fig:synthetic-k-estimation-plot}
\end{figure}

Recall the lower-dimensional synthetic distributions $\mu_k$ from \Cref{subsec: empirical-res-various-intrinsic-dim}. Here, we describe our procedure for constructing those distributions from the reprogramming dataset.

The support of $\mu$ as defined in \eqref{eq: app atomic mu} is a finite set of vectors in the simplex; it can be packaged as a matrix $M\in\mathbb{R}_{+}^{N \times d}$ whose rows are the atoms of $\mu$. Now, consider a rank-$r$ non-negative matrix factorization $M \simeq WH$, where $W\in \mathbb{R}_{+}^{N \times r}$ and $H \in \mathbb{R}_{+}^{r \times d}$. 
The rank of $M_r := WH$ is at most $r$, and we can rescale the rows of $M_r$ to the simplex without changing the rank. The result is a set of vectors in the simplex that is also contained in an $r$-dimensional linear subspace. Hence, the rows of this rescaled version of $M_r$ form the support of our at-most $r$-dimensional synthetic distribution. We compute its intrinsic dimension $k$ via the heuristic from \Cref{loc:assessing_the_intrinsic_dimension} again, noting that $k \leq r$ necessarily; see \Cref{fig:synthetic-k-estimation-plot}. Some statistics are reported in \Cref{tab: synthetic-statistics}.
\begin{table}[h]
    \centering
    \begin{tabular}{c c c c}
    \hline
         Rank~$r$ & Relative error & Intrinsic dimension $k$ & Average sparsity~$\mathbb{E}_{X \sim \mu_k}\|X\|_0$\\
         \hline
10       & 0.45           & \textcolor{blue}{3.699}   & 987                                  \\
30       & 0.30           & \textcolor{blue}{5.068}   & 976                                  \\
50       & 0.24           & \textcolor{blue}{5.871}   & 981         
            \\
200     & 0.12            & \textcolor{blue}{7.782}   & 984
            \\
            \hline
            \\
    \end{tabular}
    \caption{\textbf{Statistics of the synthetic distributions $\mu_k$}. The rank $r$ is the one used for matrix factorization, and the relative error of the factorization is defined as $\norm{\bar{M}_r-M}_F / \norm{M}_F$ where $\bar{M}_r$ is the rescaled version of $WH$. The intrinsic dimension $k$ is computed via the heuristic from  \Cref{loc:assessing_the_intrinsic_dimension}. Finally, recall that $d=1000$ is the upper bound on the average sparsities.}
    \label{tab: synthetic-statistics}
\end{table}

\subsection{Computing optimal transport}
The experiments of \Cref{sec:Empiricalresults} involve computing Wasserstein distances between finitely supported atomic measures. The Python Optimal Transport package~\cite{flamary2021POT} provides various tools for this. We used the function \texttt{emd2}, which involves no additional regularization, together with the \texttt{cityblock} distance.





\end{appendix}

\bibliographystyle{imsart-nameyear} 
\bibliography{references.bib}


%

\end{document}